\documentclass[journal]{IEEEtran}

\usepackage{graphicx}
\usepackage{calc}
\usepackage{amstext}
\usepackage{amsthm}
\usepackage{multicol}
\usepackage{times}
\usepackage{times} 
\usepackage{pdfsync} 
\usepackage{epsfig}
\usepackage{epstopdf}
\usepackage{amsmath} 
\usepackage{amssymb}
\usepackage{color} 
\usepackage{cite} 
\usepackage{array}
\usepackage{mathrsfs}
\usepackage{fix2col} 

\begin{document}
%
\title{Sequence-based Anytime Control\thanks{A
    preliminary version of parts of this
    work was
    presented  at the 49th IEEE
Conference on Decision and Control, see \cite{gupque10b}.}}
\author{Daniel E.~Quevedo,~\IEEEmembership{Member,~IEEE,}\thanks{D.\ Quevedo
    is with the School of Electrical Engineering \& Computer Science, The
    University of Newcastle, Australia: {\tt dquevedo@ieee.org}}, Vijay
  Gupta,~\IEEEmembership{Member,~IEEE}\thanks{V.\ Gupta is with 
    the Department of Electrical Engineering, University of Notre Dame: {\tt
      vgupta2@nd.edu}} }

\markboth{IEEE Transactions on Automatic Control}%
{Quevedo and Gupta: Sequence-based Anytime Control}

\maketitle


\newcommand{\eq}{\triangleq}
\newcommand{\ejw}{e^{j\omega}} 
\newcommand{\dw}{d\omega}
\newcommand{\st}{^\star} 
\newcommand{\tr}{\mathrm{tr}}
\newcommand{\diag}{\mathrm{diag}}

\newcommand{\Prob}{\mathbf{Pr}}
\newcommand{\E}{\mathbf{E}}

\newcommand{\field}[1]{\mathbb{#1}} 
\newcommand{\R}{\field{R}}
\newcommand{\N}{\field{N}} 
\newcommand{\U}{\field{U}} 
\newcommand{\X}{\field{X}}
\newcommand{\B}{\field{B}} 
\newcommand{\Z}{\field{Z}}
\newcommand{\G}{\field{G}} 
\newcommand{\Sset}{\field{S}}
\newcommand{\K}{\mathcal{K}}

\newcommand{\hfs}{\hfill\ensuremath{\square}} 
\newcommand{\bs}{\boldsymbol}
\newcommand{\uc}{\mathrm{uc}}

\newtheorem{thm}{Theorem} 
\newtheorem{defi}{Definition}
\newtheorem{coro}{Corollary} 
\newtheorem{lem}{Lemma} 
\newtheorem{rem}{Remark}
\newtheorem{ex}{Example} 
\newtheorem{ass}{Assumption} 
\graphicspath{{Matlab/}}
\allowdisplaybreaks

\begin{abstract}
  We present two related anytime algorithms for control of
  nonlinear systems  when the
  processing resources available are time-varying. The basic idea is to
  calculate tentative control input sequences for as many time steps into the
  future as allowed 
  by the available processing resources at every time step. This serves to
  compensate for the time steps when the processor is not available to perform
  any control calculations.   Using a stochastic Lyapunov function based
  approach, we analyze the 
  stability of the resulting closed loop system for   the cases when the
  processor availability can be modeled as an independent and identically
  distributed sequence and via an underlying Markov chain.  Numerical
  simulations indicate  that the increase in performance due to the proposed
  algorithms can 
  be significant. 
\end{abstract}


\section{Introduction}
A lot of recent attention has  focused on networked and embedded control
(see, e.g., the special issue~\cite{ab07} and the references therein). One
issue which 
plays an important role, especially in embedded systems, is that of time-varying
and limited processing power. As more and more objects are equipped with
micro-processors that are responsible for multiple functions such as control,
communication, data fusion, system maintenance and so on, the implicit
assumption traditionally made in control design about the processor being able to
execute the desired control algorithm at any time will break down. Similarly, if
a remote controller is in charge of many devices, multiple control tasks will compete
for shared processor resources, leading to constrained availability of
processing resources for the individual control loops. It is, thus, of interest
to study control algorithms that can function despite limited and time-varying
availability of 
processing power.  There is a growing number of works that deal with
this issue. The
impact of finite computational power has been looked at most closely for
techniques such as model predictive  control. McGovern and
Feron~\cite{MF98,MF99} presented bounds on computational time for achieving
stability for specific optimization algorithms, if the processor has constant,
but limited, computational resources. Henriksson et al~\cite{ha04,hcaa02}
studied the effect of not updating the control input in continuous time systems
for the duration of the computational delay for optimization algorithms based on
active set methods. Also related are works on event-triggered and self-triggered
control systems, and online sampling, e.g.,~\cite{t07,wl07,vmb09,cvmc10}, where a
control input is calculated aperiodically, but on demand, depending on the
plant state. In addition, we would like to mention  work on scheduling
of control tasks~\cite{cbs02,CEBA02,slss96} that looks at the problem of
processor queue scheduling, when control calculation is merely one of the tasks
in the queue.

\par An alternative approach to
achieve system robustness in the presence of
time-varying processing resources is to develop {\em anytime}
algorithms. The main
purpose of anytime algorithms is to
provide \emph{a} solution even with limited processing resources, and to refine the
solution as more resources become available. Anytime algorithms seek to make
efficient use of resources and are, thus, popular
in the context of real-time systems. 
In control, however, there are few methods available for developing
anytime controllers. A notable work is that of Bhattacharya et al~\cite{bb04}
who focused on linear systems, and presented a control algorithm that
updated a different number of states depending on the available computational
time. However, the available computational time was required to be known to the
controller a priori. Another important work is that of Greco et al~\cite{GFB07},
who proposed switching among a pre-designed set of controllers that may require
different execution times. Although the idea can be generalized to nonlinear
processes, the analysis in the paper relied on Markovian jump linear system
theory. In Gupta and Luo~\cite{g10}, an anytime algorithm for systems with multiple
inputs was presented. The main idea was based on calculating the components of the
control vector sequentially, and refining the process model as more processing
time becomes available. Since the algorithm is based on identifying the modes of the
process that require more urgent control, it is, thus, again largely limited
to linear processes.

\par In the present work, we present two anytime control algorithms for nonlinear
plants described in state-space form that are based on using extra
processor availability 
to calculate sequences which have the potential to be
implemented at the plant input at future times.  This  safeguards
 performance at those time steps where the
processor is entirely unavailable for control.  
Availability of processor time for control  calculations determines the length of
the sequences calculated and, thereby, affects the quality of the result. A distinguishing
feature of the algorithms presented is that processor availability is allowed to
be  random, with unknown distribution. Moreover, our algorithms  are one of the first  that are suitable for nonlinear plants.
 For cases where processor availability is governed by a suitable Markov
Chain,   we use Lyapunov functions
to establish sufficient conditions for  stochastic 
stability of the closed loop.  
Numerical simulations illustrate that performance gains achieved with the
algorithms proposed can be significant.

\par It is worth emphasizing that in the algorithms
presented, the potential control 
values are calculated sequentially, reutilizing the already computed values for the
next computation. This is computationally attractive, especially since the
length of the sequence to be calculated is time-varying and not known
a-priori. Thus, our approach differs significantly
from the methods used in
packetized predictive control, e.g.,
in~\cite{munchr08,tansil06,pinpar08,finvar08,quenes11a,queost11,quenes12a,quesil08}. In the
latter works calculation of control sequences requires solving
optimization problems over a finite horizon of length determined by the
controller itself.

\par The remainder of this manuscript is organized as follows: In
Section~\ref{sec:prob_form} we 
formulate the anytime control design problem studied. 
Section~\ref{sec:our_description} presents the proposed algorithms. Stochastic
stability analysis   is carried out in
Sections~\ref{sec:analysis} to~\ref{sec:mark-chain-proc}.
 Numerical
simulations are documented in
Section~\ref{sec:numerical_examples}. Section~\ref{sec:conclusions} draws the final conclusions.

\paragraph*{Notation}
\label{sec:notation}
We write $\N$ for $\{1, 2,  \ldots\}$ and $\N_0$ for $\N \cup \{0\}$. $\R$ represents the real numbers and
$\R_{\geq 0}\eq [0,\infty)$. The $p\times p$ identity matrix is
denoted via $I_p$,  $0_{p \times q}$ is the $p\times q$  all-zeroes
matrix, $0_p\eq 0_{p\times p}$, and $\mathbf{0}_p\eq 0_{p\times 1}$.  The notation
$\{x\}_{\K}$ stands for 
$\{x(k) \;\colon k \in \K\}$, where $\K\subseteq \N_0$.
We  adopt the conventions
$\sum_{k=\ell_1}^{\ell_2}a_k = 0$ and $\prod_{k=\ell_1}^{\ell_2}a_k = 1$,  if $\ell_1 > \ell_2$ and irrespective of
$a_k \in \R$. The superscript $^T$  refers to transpose. The
Euclidean norm of 
a vector $x$ is denoted via $|x|=\sqrt{x^Tx}$.
 A
function $\varphi\colon \R_{\geq 0}\to \R_{\geq 0}$ is of
\emph{class-}$\mathscr{K}_\infty$ ($\varphi \in \mathscr{K}_\infty$), if it is
continuous, zero at zero, strictly increasing, and  unbounded. The probability of an event
$\Omega$ is denoted by 
$\Prob\{\Omega \}$ and the conditional probability of $\Omega$ given
 $\Gamma$ by $\Prob\{\Omega\,|\,\Gamma \}$. The  expected value of a 
 random variable $\nu$ given 
 $\Gamma$ is denoted by  $\E\{\nu  \,|\, \Gamma \}$ while $\E\{\nu\}$ refers 
 to the
 unconditional 
 expectation. 

\section{Problem Formulation}
\label{sec:prob_form}
We consider
 nonlinear (and possibly unstable) plants sampled periodically with sampling interval
  $T_s>0$ and described in discrete-time via:
\begin{equation}
  \label{eq:process}
  x(k+1) = f(x(k),u(k),w(k)),\quad k\in\N_0.
\end{equation}
In~\eqref{eq:process}, $x\in \R^n$ is the
   plant state,  $u\in \R^p$ is the plant input, and  $w\in\R^m$  is an
   unmeasured disturbance. The 
model~\eqref{eq:process} satisfies  
 $f(\mathbf{0}_n,\mathbf{0}_p,\mathbf{0}_m)=\mathbf{0}_n$ and the initial state, $x(0)$, is arbitrarily distributed
(with possibly unbounded support). 

\par Throughout this work, we will assume that the unperturbed plant model 
\begin{equation}
      \label{eq:15}
      x(k+1) =  f(x(k),u(k),\mathbf{0}_m)
    \end{equation} 
is globally stabilizable via state feedback. More precisely, we make the
following assumption: 
\begin{ass}
\label{ass:CLF}
There exist functions $\varphi_1, \varphi_2\in\mathscr{K}_\infty$, $V\colon
\R^n\to\R_{\geq 0}$,   $\kappa \colon \R^n\to \R^p$, and a constant $\rho \in [0,1)$,
 such that for all $x\in\R^n$,
\begin{equation}
  \label{eq:3}
  \begin{split}
    \varphi_1(|x|)\leq V(x)&\leq \varphi_2(|x|)\\
    V(f(x,\kappa(x),\mathbf{0}_m)) &\leq \rho V(x).
  \end{split}
\end{equation}
 \hfs
\end{ass}

When implementing discrete-time control systems it is generally  assumed that
the processing resources available to the
controller are such that the  
control law can always be evaluated within
a fixed (and small) time-delay, say $\delta\in (0,T_s)$.\footnote{Recall that
  fixed delays can be easily incorporated into the model~(\ref{eq:process}) by aggregating
  the previous plant input to the plant state, see also\cite{nilber98}. For ease
  of exposition, throughout this work, we will 
 use the standard discrete-time notation as in~(\ref{eq:process}).} However, in practical 
networked and embedded systems, the processing resources (e.g., processor execution times)  
available for control calculations may vary and, at times, be insufficient to
generate a control input within the prescribed time-delay $\delta$. 
 One possible
remedy for this issue would to redesign the control system for a worst case by
choosing larger values of $\delta$ and, possibly, $T_s$. Clearly, such an approach
will, in general, lead to unnecessary conservativeness and associated poor
performance. In the present work we adopt an anytime control
paradigm to seek  favorable trade-offs between processor availability and
control performance.

\par Before proceeding we note that a direct implementation of the control policy used in
Assumption~\ref{ass:CLF}, when processing 
resources are time varying,  results in a {\em baseline}
algorithm, which gives rise to the plant input: 
\begin{equation}
  \label{eq:4}
  u(k)=
  \begin{cases}
    \kappa(x(k)), &\text{if sufficient computational resources to}\\
       &\text{evaluate $\kappa(x(k))$ are available
      between}\\
      &\text{ $kT_s$ and $kT_s+\delta$,}\\
    \mathbf{0}_p, &\text{otherwise,}
  \end{cases}
\end{equation}
where the symbol
 $u(k)$ with $k\in\N_0$ denotes  the plant input which is applied during the  interval\footnote{If
  sufficient computational 
  resources are not available, then one could alternatively hold the previous
  control value and set $u(k) = u(k-1)$. The situation mirrors that encountered when the
  control input is affected by
   dropouts; see, e.g., \cite{s09}.}
 $[kT_s +\delta, (k+1)T_s +\delta)$.
Whilst the baseline algorithm~(\ref{eq:4}) is  intuitive and simple,
it is by no means clear that it cannot be outperformed by
more elaborated control formulations. In the following section, we will present two related anytime control algorithms for
the plant model~\eqref{eq:process}. The aim is to make more efficient use of the
processing resources
available for control, when compared to the baseline algorithm~(\ref{eq:4}).


\section{Anytime Control Through Calculation of Control Sequences}
\label{sec:our_description}
 Throughout this work, we will assume that the controller needs processor time to carry out
mathematical computations, 
such as evaluating functions. However, simple operations at a bit level, such as writing data into buffers, shifting buffer
contents and setting values to zero, do not require processor time. Similarly, 
input-output operations, i.e.,  
A/D and D/A conversion are triggered by external asynchronous loops with a
real-time clock and 
do not require that the processor be available for control. As in regular
discrete-time control, these external loops
ensure that state measurements are available at the instants
$\{kT_s\}_{k\in\N_0}$ and that the controller outputs are passed on to the plant
actuators at times $\{kT_s+\delta\}_{k\in\N_0}$, where $\delta$ is fixed; see, e.g., \cite{astwit90}

\par A standing assumption is that if  the processor were
fully available for control, then calculating the desired plant input $u(k)$ for
a given plant 
state $x(k)$ would be possible within the pre-allocated time-frame
$t\in (kT_s,kT_s+\delta)$. Issues arise when, at times,  processor  availability
does not permit the desired plant input to be calculated. To take care of the associated
performance loss, in the present work we propose to use one
of the   two anytime control algorithms presented below. 

\subsection{Algorithm Descriptions}
\label{sec:algor-descr}
Both algorithms are based on the following basic idea: At time intervals when the 
controller is provided with more processing resources than are needed to evaluate
the current control input, the algorithm  calculates a  sequence of tentative
future plant  inputs, say $\vec{u}(k)$.  The sequence is stored in a local
buffer and may be  used 
when, at some future time 
steps, the  processor  
availability precludes any control calculations, see Fig.~\ref{fig:algorithm}.

\par For further reference, we denote the buffer states via $\{b\}_{\N_0}$, where 
\begin{equation}
  \label{eq:12}
  b(k) =
  \begin{bmatrix}
    b_1(k)\\b_2(k)\\\vdots \\ b_{\Lambda}(k)
  \end{bmatrix}\in\R^{\Lambda  p},\quad k\in\N_0
\end{equation}
for a given value $\Lambda\in\N$ and where each $b_j(k)\in\R^p$, $j\in\{1,\dots,\Lambda\}$.
We also  
introduce the shift matrix $S$ and the unit vector $e_1$ via:
 \begin{align}
\label{eq:21}
    S&\eq
\begin{bmatrix}
    0_p & I_p& 0_p &\hdotsfor{1} &0_p\\
    \vdots & \ddots & \ddots &\ddots  & \vdots\\
     0_p &  \dots      &  0_p &I_p  & 0_p\\
    0_p & \hdotsfor{2}       &  0_p & I_p\\
 0_p &\hdotsfor{3} &  0_p
\end{bmatrix}\in\R^{\Lambda  p\times \Lambda p},\\
\nonumber
e_1 &\eq
\begin{bmatrix}
  I_p\\0_p\\ \vdots \\0_p
\end{bmatrix}
\in\R^{\Lambda  p\times  p}.
\end{align}

\begin{figure}[t]
  \centering
  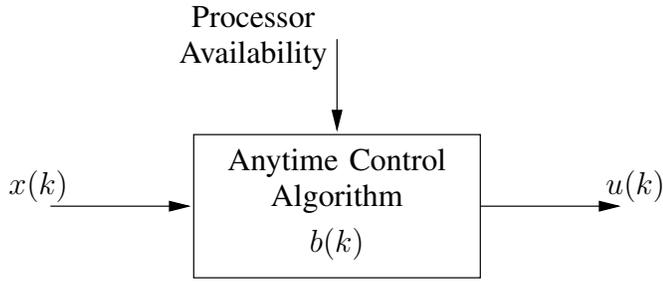
  \caption{Anytime control structure with internal buffer state $b(k)$.}
  \label{fig:algorithm}
\end{figure}

\linespread{1.3}
\begin{figure}[h!]
\noindent\rule{\linewidth}{0.2mm}\vspace{-2mm}
\begin{enumerate}[\setlabelwidth{Step 5}]
\item[Step 1:]  At time $t=0$, 
  \par\hspace{1cm} \textsc{set}  $b(-1)\leftarrow \mathbf{0}_{\Lambda  p}$, $k\leftarrow 0$
\item[Step 2:] \label{step:timek}
  \textsc{if}  $t \geq k T_s$,  \textsc{then}
  \par\hspace{1cm}  \textsc{input} $x(k)$;
  \par\hspace{1cm}  \textsc{set}  $\chi\leftarrow x(k)$, $j\leftarrow 1$,
   $b(k)\leftarrow Sb(k-1)$;\\ 
  \textsc{end}

\item[Step 3:] \label{step:repeat}
  \textsc{while} ``sufficient processor time is available'' 
  and  $j\leq \Lambda$ and time $t < (k+1) T_s$,
  \par\hspace{1cm} \textsc{set} $v \leftarrow V(\chi)$, 
  where $V$ is the Lyapunov  function in~(\ref{eq:3});
  \par \hspace{1cm}  Use $v$ and $\chi$ to find $u_j(k)$, such that
  \begin{equation}
    \label{eq:14}
    V(f(\chi,u_j(k),\mathbf{0}_m))\leq \rho v;
  \end{equation}
  
  \par\hspace{1cm} \textsc{if} $j=1$, \textsc{then}
  \par \hspace{2cm} \textsc{output} $u_{1}(k)$;
  \par \hspace{2cm} \textsc{set} $b(k)\leftarrow\mathbf{0}_{\Lambda  p}$;
  \par \hspace{1cm} \textsc{end}
  
  \par \hspace{1cm} \textsc{set} $b_j(k)\leftarrow u_j(k)$;
  
  \par \hspace{1cm} \textsc{if} ``sufficient processor time is not available'' or
  $t \geq (k+1) T_s$, \textsc{then}
  \par \hspace{2cm} \textsc{goto} Step 5
  \par \hspace{1cm} \textsc{end}
  
  \par \hspace{1cm} \textsc{set} $\chi \leftarrow  f(\chi,u_j(k),\mathbf{0}_m)$,
  $j\leftarrow j+1$;\\
  \textsc{end}
  
\item[Step 4:] 
  \textsc{if} $j=1$, \textsc{then}
  \par\hspace{1cm} \textsc{output} $b_{1}(k)$;\\
  \textsc{end}

\item[Step 5:] 
  \textsc{set} $k \leftarrow k+1$ and \textsc{goto} Step 2;
\end{enumerate}
\vspace{-3mm}
\noindent\rule{\linewidth}{0.2mm}
\caption{Algorithm A$_1$}
\label{alg:1}
\end{figure}
\linespread{1}

\par Algorithm A$_1$ is presented in Fig.~\ref{alg:1}. It can be seen that the 
algorithm proposed amounts to a dynamic state
feedback policy with internal state variable $b(k)$. The latter provides 
\begin{equation}
  \label{eq:17}
  u(k) = e_1^T b(k)=b_{1}(k),
\end{equation}
 and suggested plant 
inputs at future time steps. At the time steps when more processor time is
available, a longer suggested trajectory of  control inputs is calculated and
stored in the buffer.\footnote{Note that, by
    Assumption~\ref{ass:CLF}, in Step 3, one could simply set $u_j(k)\leftarrow
    \kappa(\chi)$.}  If the buffer runs out of tentative plant inputs (as
calculated in Step 3), then  the actuator values are set to zero.
 With Algorithm $\textsc{A}_{1}$, as soon as the processor
calculates a control input $u_0(k)$, it throws away the remaining elements in
the buffer, see  line ``$b(k)\leftarrow\mathbf{0}_{\Lambda   p}$'' in Step 3. 

\linespread{1.3}
\begin{figure}
\noindent\rule{\linewidth}{0.2mm}\vspace{-2mm}
\begin{enumerate}[\setlabelwidth{Step 3}]
\item[Step 3:] \label{step:repeat2}
  \textsc{while} ``sufficient processor time is available'' and  $j\leq \Lambda$ and
  time $t < (k+1) T_s$,
  
  \par\hspace{1cm} \textsc{set} $v \leftarrow V(\chi)$,  where $V$ is the Lyapunov
  function in~(\ref{eq:3});
  
  \par\hspace{1cm}   Use $v$ and $\chi$ to find $u_j(k)$, such that
  $V(f(\chi,u_j(k),0))\leq \rho v$;
  
  \par\hspace{1cm} \textsc{if} $j=1$, \textsc{then}
  \par\hspace{2cm} \textsc{output} $u_{1}(k)$;
  \par\hspace{1cm} \textsc{end}
  
  \par\hspace{1cm} \textsc{set} $b_j(k)\leftarrow u_j(k)$;
  
  \par\hspace{1cm}  \textsc{if} ``sufficient processor time is not available'' or
  $t \geq (k+1) T_s$, \textsc{then}
  \par\hspace{2cm} \textsc{goto} Step 5;
  \par\hspace{1cm} \textsc{end}
  
  \par\hspace{1cm} \textsc{set} $ \chi \leftarrow  f(\chi,u_j(k),\mathbf{0}_m)$,
  $j\leftarrow j+1$;\\
  \textsc{end}
\end{enumerate} 
\vspace{-3mm}
\noindent\rule{\linewidth}{0.2mm}
\caption{Step 3 of Algorithm A$_2$}
\label{alg:2}
\end{figure}
\linespread{1}

\par Algorithm $\textsc{A}_{2}$ is almost identical to the first algorithm,
$\textsc{A}_{1}$. The only 
difference is that, in Step 3, the buffer contents
are never re-set to zero, i.e., the line ``$b(k)\leftarrow\mathbf{0}_{\Lambda 
  p}$''
 is eliminated, see Fig.~\ref{alg:2}. Thus, if Algorithm $\textsc{A}_{2}$ is
 used, then buffer
elements may stem from calculations carried out at different time instants.
By not deleting the entire buffer, but only
replacing the \emph{appropriate} entries, when using $\textsc{A}_2$ the buffer will
run out of data less often than when using $\textsc{A}_1$. 

\par It is worth noting that neither algorithm requires prior knowledge of future
 processor  availability for control.  This opens 
the possibility to  employ the algorithms  in shared systems, where the controller 
task can be preempted by other computational tasks carried out by the processor,
see also \cite{cacbak05,lhll05,zhs99}.  As in  other anytime algorithms, there exists a compromise between resultant
 closed loop
 performance and the processor availability. Understanding this trade-off forms
 the bulk of this work.


\subsection{Basic Properties}
\label{sec:basic-properties}
With the algorithms presented in Section~\ref{sec:algor-descr},  extra
processing time is used to calculate additional elements of the 
tentative plant input sequences, thus, providing 
higher quality results, i.e.,  sequences $\vec{u}(k)$ which better safeguard 
  against performance loss at \emph{future} time instances where  processor
  availability may be insufficient.
To further elucidate the situation, we note that in both algorithms, during each iteration of
the while-loop in Step~3, the state value $x(k)$ is used to calculate a
tentative control, namely $u_j(k)$. In the sequel, we will 
denote by
$N(k)$ the total number of iterations of the while-loop which are carried out during the
interval $t\in (kT_s,(k+1)T_s)$ and note that $N(k)\in\{0,1,\dots,\Lambda\}$. Thus,
if $N(k)\geq 1$, then the entire sequence of tentative controls is
\begin{equation}
  \label{eq:4b}
  \vec{u}(k) = 
  \begin{bmatrix}
    u_{1}(k)\\u_{2}(k)\\ \vdots\\ u_{N(k)}(k)
  \end{bmatrix}\in\R^{N(k)\cdot p}.
\end{equation}
If $N(k)=0$, then the processor was not available for  control, and (with {either of the}
algorithms) the actuator values are
taken as the first $p$ elements of the shifted state $b(k)=Sb(k-1)$, see~\eqref{eq:21}.

\par In terms of the notation introduced above and in~(\ref{eq:21}), if Algorithm A$_1$ is used, then the
buffer $b(k)$ obeys the recursion: 
\begin{equation}
  \label{eq:1}
  b(k)=
  \begin{cases}
    Sb(k-1), &\text{if $N(k)=0$,}\\
    \begin{bmatrix}
      \vec{u}(k)\\
      \mathbf{0}_{(\Lambda -N(k)) p}
    \end{bmatrix},
&\text{if $N(k)\geq 1$.}
  \end{cases}
\end{equation}
On the other hand, if Algorithm A$_2$ is used, then  we have:
\begin{equation}
  \label{eq:18}
  b(k)=
  \begin{cases}
    Sb(k-1), &\text{if $N(k)=0$,}\\
    \begin{bmatrix}
      \vec{u}(k)\\
      \mathbf{0}_{(\Lambda -N(k)) p}
    \end{bmatrix}
    + M_{N(k)} b(k-1),
    &\text{if $N(k)\geq 1$,}
  \end{cases}
\end{equation}
where
\begin{equation}
  \label{eq:44}
  M_i\eq \big(I_{\Lambda  p}- D_i \big) S,
\end{equation}
with
\begin{equation}
\label{eq:45}
    D_i \eq
    \begin{cases}
      \diag ( I_{i p}, 0_{(\Lambda-i) p}), &\text{if $i\in
        \{1,2,\dots,\Lambda-1\}$}\\     
       I_{\Lambda p}, &\text{if $i = \Lambda$.}
     \end{cases}
   \end{equation}

\par In addition to studying the length of the tentative control sequences provided by the
algorithms, namely $\{N\}_{\N_0}$, it is convenient to investigate how many 
values which stem from the tentative control sequences $\{\vec{u}(k-\ell)\}$,
$\ell \in\N_0$ are contained in the
buffer state $b(k)$. We will refer to this value as the \emph{effective buffer
  length} (at time $k$), and denote it as
\begin{equation}
  \label{eq:5}
  \lambda (k) \in\{0,1,\dots,\Lambda\}, \quad k\in\N_0
\end{equation}
with $\lambda(-1)=0$. It is easy to see that, if
Algorithm $\textsc{A}_1$ is used, then 
$\{\lambda\}_{\N_0}$ is governed by
\begin{equation}
  \label{eq:19}
  \lambda(k)=
  \begin{cases}
    N(k), &\text{if $N(k)\geq 1$},\\
    \max\{\lambda(k-1)-1,0\},  &\text{if $N(k)=0$},
  \end{cases}
\end{equation}
whereas, with Algorithm $\textsc{A}_2$,
we have
\begin{equation}
  \label{eq:9}
  \lambda(k) = \max\{N(k), \lambda(k-1)-1\},\quad k\in \N_0.
\end{equation}
The following example illustrates the quantities introduced above:

\begin{ex}
\label{ex:one}
 Suppose that $\Lambda = 5$ and that the processor availability for control is such that
$N(0)=5$, $N(1)=0$, $N(2)=1$, $N(3)=0$.
 If Algorithm A$_1$ is used, then the buffer state at times $k\in\{0,1,2,3\}$ becomes:
\begin{align*}
  b(0)=
    \begin{bmatrix}
    u_{{1}}(0)\\u_{{2}}(0)\\ u_{{3}}(0)\\ u_{{4}}(0)\\u_{{5}}(0)
  \end{bmatrix},&\qquad b(1)=
\begin{bmatrix}
    u_{{2}}(0)\\ u_{{3}}(0)\\ u_{{4}}(0)\\u_{{5}}(0)\\ \mathbf{0}_{p}
  \end{bmatrix},\\ b(2)=
\begin{bmatrix}
    u_{{1}}(2)\\ \mathbf{0}_{p}\\ \mathbf{0}_{p}\\ \mathbf{0}_{p} \\ \mathbf{0}_{p}
  \end{bmatrix},&\qquad b(3)=
\begin{bmatrix}
    \mathbf{0}_{p} \\ \mathbf{0}_{p} \\  \mathbf{0}_{p}\\ \mathbf{0}_{p}\\ \mathbf{0}_{p}
  \end{bmatrix},
\end{align*}
%
%
which gives $\lambda(0)=5$, $\lambda(1)=4$, $\lambda(2)=1$, $\lambda(3)=0$,  and  plant inputs
$u(0)= u_{{1}}(0)$, $u(1)=u_{{2}}(0)$, $u(2)=u_{{1}}(2)$, $u(3)=\mathbf{0}_{p}$.
 On the other hand, if Algorithm A$_2$ is used, then we have
\begin{align*}
  b(0)=
    \begin{bmatrix}
    u_{{1}}(0)\\u_{{2}}(0)\\ u_{{3}}(0)\\ u_{{4}}(0)\\u_{{5}}(0)
  \end{bmatrix}, &\qquad b(1)=
\begin{bmatrix}
    u_{{2}}(0)\\ u_{{3}}(0)\\ u_{{4}}(0)\\u_{{5}}(0)\\ \mathbf{0}_{p}
  \end{bmatrix},\\b(2)=
\begin{bmatrix}
    u_{{1}}(2)\\ u_{{4}}(0)\\ u_{{5}}(0)\\ \mathbf{0}_{p} \\ \mathbf{0}_{p}
  \end{bmatrix},&\qquad b(3)=
\begin{bmatrix}
    u_{{4}}(0) \\ u_{{5}}(0) \\  \mathbf{0}_{p}\\ \mathbf{0}_{p}\\ \mathbf{0}_{p}
  \end{bmatrix} 
\end{align*}
and $\lambda(0)=5$, $\lambda(1)=4$, $\lambda(2)=3$, $\lambda(3)=2$,  
$u(0)= u_{{1}}(0)$, $u(1)=u_{{2}}(0)$, $u(2)=u_{{1}}(2)$, $u(3)=u_{{4}}(0)$.
%
Note that with Algorithm
A$_1$, $\lambda(3)=0$ and therefore  the plant input at time $k=3$ is set to
zero; with Algorithm A$_2$, the  
value calculated at time $k=0$ is used.  \hfs
\end{ex}

\par 

\section{Stochastic Stability of Anytime Control Algorithms}
\label{sec:analysis}
 Since the  processor availability  for control calculations is
random, the plant input is random, and thus the 
system~(\ref{eq:process}) evolves stochastically.  Various stability notions for
stochastic systems have been studied in the literature
(e.g.,~\cite{jcfl91,k71}). In the present work, we are interested in the following notion:
\begin{defi}[Stochastic Stability]
  A dynamical system with state trajectory $\{x\}_{\N_0}$ is said to be stochastically stable, if
  \begin{equation}
    \label{eq:33}
    \sum_{k={0}}^{\infty}\E\big\{\varphi(|x(k)|)\big\} <\infty,
  \end{equation}
  for some $\varphi \in\mathscr{K}_\infty$.\hfs
\end{defi}

\begin{rem}
  \label{rem:stabnotions}
  It follows directly from~(\ref{eq:33}), that
  stochastic stability implies that
  there exists  $\varphi \in\mathscr{K}_\infty$, such that:
  \begin{equation}
    \label{eq:34}
    \lim_{k\to\infty} \E \big\{\varphi(|x(k)|)\big\} =0.
  \end{equation}
 In the particular case where $\varphi(s)=s^2$,~(\ref{eq:33})
reduces to $\sum_{k={0}}^{\infty}\E \{|x(k)|^2\} <\infty$, and~(\ref{eq:34}) to
 $\lim_{k\to\infty} \E \{|x(k)|^2\} =0$; see also~\cite{fanlop02,jcfl91}. 
\hfs
\end{rem}

\subsection{Assumptions}
\label{sec:assumptions}
 Our subsequent stability analysis considers the 
unperturbed system~\eqref{eq:15}, i.e., 
where $w(k)=0$, for all $k\in\N_0$. For pedagogical ease, we also begin by presenting the analysis with the additional assumption 
that the processor availability for control is  independent and
  identically distributed (i.i.d.). Thus, for the analysis in 
 Sections~\ref{sec:algorithm-a_1} and~\ref{sec:algorithm-a_2}, we make the following assumption:

\begin{ass}
\label{ass:iid}
  The process $\{N\}_{\N_0}$ introduced in Section~\ref{sec:basic-properties} is
  i.i.d., with probability distribution
  \begin{equation}
    \label{eq:7}
    \Prob \{N(k)=l\} =p_l,
  \end{equation}
  where $l\in\{0,1,2,\dots, \Lambda\}$ and with $p_0\in[0,1)$.\hfs
\end{ass}
In Section~\ref{sec:mark-chain-proc}, we will show how to extend this analysis for the case when the processor availability can be described by a Markov chain, and thus has memory.

\par Assumption~\ref{ass:bound_prob} stated below, bounds the rate of increase
of the Lyapunov function $V$ in~(\ref{eq:3}), when the 
  nominal system~\eqref{eq:15} is run in open-loop. It also imposes a (mild) restriction on
  the distribution of the  initial plant state.
\begin{ass}
  \label{ass:bound_prob}
  There exists $\alpha \in [1,1/p_0)$ 
such that
  \begin{equation}
    \label{eq:20}
    V({f}(x,\mathbf{0}_p,\mathbf{0}_m))\leq\alpha V(x),\quad\forall x \in\R^n.
  \end{equation}
The initial plant state satisfies
\begin{equation}
  \label{eq:66}
  \E\big\{\varphi_2(|x(0)|)\big\}<\infty,
\end{equation}
 where
$\varphi_2\in\mathscr{K}_\infty$ is as in~(\ref{eq:3}).\hfs
\end{ass}
It is worth emphasizig that the fact that Assumptions~\ref{ass:CLF}
and~\ref{ass:bound_prob} are global and stated in terms of a common Lyapunov
function limits the class of plants and control policies considered in our
subsequent analysis.  One case where~(\ref{eq:20}) is {
satisfied} is when $V$ and $f$ are globally Lipschitz continuous, more precisely,
when there exist $\varphi_V,\varphi_f\in\R_{\geq 0}$ such that:
\begin{equation*}
  \begin{split}
    |V(x)-V(z)|&\leq \varphi_V |x-z|\\
    |f(x,u,w)-f(z,u,w)|&\leq \varphi_f |x-z|.
  \end{split}
\end{equation*}
In this case, and since
${f}(\mathbf{0}_n,\mathbf{0}_p,\mathbf{0}_m)=\mathbf{0}_n$, we have
\begin{align*}
  \label{eq:2}
    V({f}(x,\mathbf{0}_p,\mathbf{0}_m))&=| V({f}(x,\mathbf{0}_p,\mathbf{0}_m))
    -V({f}(\mathbf{0}_n,\mathbf{0}_p,\mathbf{0}_m))|\\&\leq
    \varphi_V |{f}(x,\mathbf{0}_p,\mathbf{0}_m)| \\
    &\leq \varphi_V
    \varphi_f|x|\leq \varphi_V
    \varphi_f \varphi^{-1}_1V(x)= \alpha V(x),
\end{align*}
for $\alpha=(\varphi_V\varphi_f)/\varphi_1$ and~(\ref{eq:20}) will hold provided
$p_0<\varphi_1/(\varphi_V\varphi_f)$.  

\begin{ex}
  Consider an open-loop unstable constrained plant model  of the
  form~(\ref{eq:process}) with
\begin{equation*}
  f(x,u,w) =
  \begin{bmatrix}
    x_2+u_1\\-{\rm{sat}} (x_1+x_2) +u_2
  \end{bmatrix}+
\begin{bmatrix}
 \sqrt{w^2+5} -\sqrt{5}\\
  0
\end{bmatrix},
\end{equation*}
with
\begin{equation*}
x=
\begin{bmatrix}
  x_1\\x_2
\end{bmatrix},\;
u=
\begin{bmatrix}
  u_1\\u_2
\end{bmatrix},\quad
  \rm{sat}(\mu)=
  \begin{cases}
    -1,&\text{if $\mu<-1$,}\\
    \nu & \text{if $\mu \in [-1,1]$,}\\
    1,&\text{if $\mu>1$},
  \end{cases}
\end{equation*}
see\cite[Example 2.3]{khalil96} and \cite{quenes12a}.  The second component of the  plant input is
constrained via 
$|u_2(k)|\leq 0.8$, $\forall k\in\N_0$. If we choose
$V(x)=2|x|$  and   policy $\kappa(x)=\big[\begin{matrix}
   - x_2& 0.8 {\rm{sat}} (x_1+x_2)
  \end{matrix}\big]^T$, then direct calculations provide that
\begin{equation*}
 \begin{split}
   V&\big(f(x,\kappa(x),\mathbf{0}_m) \big) =
  0.4 | {\rm{sat}} (x_1+x_2)| \leq  0.4 | x_1+x_2|\\
  &\leq 0.8\max\{|x_1|,|x_2|\}-\max\{|x_1|,|x_2|\}+|x|\leq |x|.
\end{split}
\end{equation*}
Thus, Assumption~\ref{ass:CLF} holds with $\rho = 1/2$, and
$\varphi_1(s)=\varphi_2(s)=2s$. Furthermore, by proceeding as
in\cite[p.73]{khalil96}, it can be shown that~(\ref{eq:20}) 
holds with
$\alpha =1.618$. Thus, provided that~(\ref{eq:66}) holds and $p_0 < 0.618$,
Assumption~\ref{ass:bound_prob} is also satisfied. \hfs
\end{ex}

\par The following example illustrates that, at times,
it may be convenient to first find a Lyapunov function $V$ which
satisfies~(\ref{eq:20}) and then seek a control policy which ensures that
Assumption~\ref{ass:CLF} holds.  
\begin{ex}
\label{ex:2}
  Consider a  scalar unconstrained and unperturbed open-loop unstable  non-linear plant where
  $f(x,u,w)=x^2+u$, with $x,u \in\R$. A 
  stabilizing control 
  policy which 
  satisfies Assumption~\ref{ass:CLF} for $V_1(x)=|x|^2$ is 
  given by $\kappa_1(x)=-x^2+\rho x$, with $\rho\in [0,1)$. However,
  $V_1(x)=|x|^2$ is 
  not suitable for use in Assumption~\ref{ass:bound_prob}, since
  $V_1(f(x,0,0))/V_1(x)   
  =x^2\to \infty$ as $x\to\infty$.
  \par In contrast, if we choose  $V_2\in\mathscr{K}_\infty$ as
  $V_2(x)=\ln(|x|+1)$, for all $x\in\R$, then 
  \begin{multline*}
    V_2(f(x,0,0)) = \ln(x^2+1)\leq  \ln(x^2 + 2|x|+1)\\= 2 \ln(|x|+1) = 2 V_2(x), \quad
    \forall x\in\R
  \end{multline*}
  and~(\ref{eq:20}) holds with  open-loop rate of growth bound constant $\alpha
  =2$. The associated control policy $\kappa_2(x) = 
  -x^2 + \exp(\rho V_2(x)) -1$, where $\rho \in [0,1)$, gives
  \begin{align*}
    V_2(f(x,\kappa_2(x),0)) &= \ln(|\exp(\rho V_2(x)) -1|+1) \\&= \ln(\exp(\rho
    V_2(x))) =\rho V_2(x).
  \end{align*}
  We conclude that if $p_0<1/2$ and the initial plant state is suitably
  distributed, then Assumptions~\ref{ass:CLF} and~\ref{ass:bound_prob} will hold.\hfs
\end{ex}

\subsection{Stochastic Stability with  the Baseline Algorithm}
\label{sec:baseline-algorithm}
 We will next present  sufficient conditions under
which the baseline algorithm~(\ref{eq:4})  achieves stochastic
stability of the closed loop {system}. As in~(\ref{eq:7}), we denote via $p_0$ the probability that the
controller is unable to calculate any control input. 
 Thus, if the baseline
algorithm~(\ref{eq:4}) is used and Assumption~\ref{ass:iid} holds, then (in the
disturbance-free case) the closed loop is characterised by:
\begin{multline}
  \label{eq:10}
  \Prob\big\{x(k+1) = \chi^+ \,\big|\, x(k)=\chi\big\}\\=
  \begin{cases}
   p_0, &\text{if $\chi^+= f(\chi,\mathbf{0}_p,\mathbf{0}_m)$,}\\
   1-p_0, &\text{if $\chi^+=f(\chi,\kappa(\chi),\mathbf{0}_m)$}.
  \end{cases}
\end{multline}

It can be seen that the plant state trajectory is similar to that of a networked control system in which the
controller is unable to communicate with the actuator with  probability $p_{0}$
at any time step. Stability conditions for such systems have been derived both
for linear systems~\cite{gm09,i09} and nonlinear
systems~\cite{quenes12a}. In particular, for a scalar linear
  plant model with a scalar input, $$f(x,u,w)=ax + b_uu + b_w w,\quad 
  (a,b_u,b_w)\in\R^3,$$ and quadratic Lyapunov function,
  $V(x)=x^2$, the condition $p_0|a|^{2}<1$  has been  
  shown to be necessary and sufficient for stabilizability in~\cite{gm09}. Thus, the constant $\alpha$ needs to satisfy $\alpha\in[1,1/p_{0})$ for stability with the baseline algorithm.
{More generally, we have} the following
sufficient condition for stochastic 
stability when the baseline algorithm is used:
\begin{thm}
\label{thm:baseline}
  Suppose that
  Assumptions~\ref{ass:CLF} to~\ref{ass:bound_prob} hold. If 
  \begin{equation}
    \label{eq:baseline_stability}
    p_{0}\alpha+(1-p_{0})\rho < 1,
  \end{equation}
then~(\ref{eq:10}) is stochastically stable.
\end{thm}
\begin{proof}
  First we note that, by~(\ref{eq:10}), the process $\{x\}_{\N_0}$
  is Markovian. Thus, stability can be examined by using a
  stochastic Lyapunov function approach; see, e.g.,\cite{k71}. The  law of total
  expectation, when applied to $\E\{V(x(1))\,|\,x(0)\}$, with $V$ as in~(\ref{eq:3}), gives
  \begin{align}
\nonumber       &\E\big\{V(x(1))\,\big|\,x(0)=\chi\big\}\\\nonumber&=p_0 
    V(f(\chi,\mathbf{0}_p,\mathbf{0}_m))+(1-p_{0})V(f(\chi,\kappa(\chi),\mathbf{0}_m))
      \\
    &\leq
    p_{0}\alpha V(\chi)+(1-p_{0})\rho V(\chi)<V(\chi), \quad \forall \chi \in \R^n,
    \label{eq:38}
    \end{align}
where we have used~(\ref{eq:3}),~(\ref{eq:20}) and~(\ref{eq:baseline_stability}).
  Theorem 2 of \cite[Chapter 8.4.2]{k71} implies that there exists $c<\infty$
  such that $\sum_{k=0}^{\infty} \E\big\{V(x(k))\,\big|\,x(0)=\chi\big\}\leq c V(\chi)$.
  Thus, by using~(\ref{eq:3}) and taking expectation with respect to the distribution of
$x(0)$, we obtain
\begin{multline*}
 \sum_{k=0}^{\infty} \E\big\{\varphi_1(|x(k)|)\big\}
 = \E\bigg\{\sum_{k=0}^{\infty} \E\big\{\varphi_1(|x(k)|)\,\big|\,x(0)\big\}\bigg\}\\\leq
 \E\{cV(x(0))\} \leq c\E\{\varphi_2(|x(0)|)\} <\infty,
\end{multline*}
where the last inequality follows from~(\ref{eq:66}). Since
$\varphi_1\in\mathscr{K}_{\infty}$, stochastic stability
follows.
\end{proof}

For the proposed anytime algorithms, stability analysis is more subtle than for
the baseline algorithm. The  main reason is that, due to buffering, the plant state
$\{x\}_{\N_0}$ will in general not be  Markovian and simple conditioning
as in~(\ref{eq:38}) is not possible.\footnote{Note that some
 of the results included in Section IV of \cite{gupque10b} are incorrect.}

\section{Stability with Algorithm A$_1$}
\label{sec:algorithm-a_1}
To derive sufficient conditions for stochastic stability when Algorithm
$\textsc{A}_{1}$ is used, we will employ a
technique which is roughly based on the approaches used
in~\cite{quenes12a,queost11,quenes11a,kt69,xx09}. As will become apparent, randomness
of the sequence length process
$\{N\}_{\N_0}$, see~(\ref{eq:4b}),  makes
the analysis of the anytime algorithms studied significantly more involved than
the analysis of the predictive networked control formulations
of~\cite{quenes12a,queost11,quenes11a}. 

\subsection{Plant model at times $k\in\mathcal{K}$}
\label{sec:system-model-at}
 For ease of exposition, in the
sequel we assume that $N(0)>0$ and denote  the time steps at which at least one control
input is calculated via $\mathcal{K}=\{k_i\}_{i\in\N_0}$, where $k_{0}=0$ and
\begin{equation}
  \label{eq:21a}
  k_{i+1} = \inf \big\{ k\in\N \colon k>k_i,\;  N(k)>0\big\},\quad i\in\N_0.
\end{equation}

\par It is convenient to introduce the iterated mappings with input $x(k_i)$,
$k_i\in\mathcal{K}$:\footnote{For example, we have 
  ${f}^1(x(k_i))=  {f}(x(k_i),u_{0}(k_i),\mathbf{0}_m)$ and
${f}^2(x(k_i)) = {f}(
{f}(x(k_i),u_{0}(k_i),\mathbf{0}_m),u_{1}(k_i),\mathbf{0}_m)$. Note that, by Step 3 in the algorithm description, the values
  $\{{u}_j(k_i)\}$, $j\in\{0,1,\dots,N(k_i)-1\}$ are determined by $x(k_i)$.}
\begin{equation}
  \label{eq:24}
  {f}^j(x(k_i))\eq
  \begin{cases}
    x(k_i), & \text{if $j=0$},\\
    {f}({f}^{j-1}(x(k_i)),u_{j-1}(k_i),\mathbf{0}_m),&\\\qquad\qquad\qquad\text{if $j\in\{1,\dots,N(k_i)\}$}
  \end{cases}
\end{equation}
and  the related mappings which describe the nominal plant model when the input
is set to zero:
\begin{equation}
  \label{eq:27}
  {f}_{\textsc{ol}}^j(x)\eq
  \begin{cases}
    x, & \text{if $j=0$},\\
    {f}({f}_{\textsc{ol}}^{j-1}(x),\mathbf{0}_p,\mathbf{0}_m),&\text{if $j\geq 1$}.
  \end{cases}
\end{equation}
We also denote the time between two consecutive elements  of $\mathcal{K}$ via
\begin{equation}
  \label{eq:26}
  \Delta_i\eq k_{i+1}-k_i, \quad \forall (k_{i+1}, k_i) \in\mathcal{K} \times\mathcal{K}
\end{equation}
and note that, by Assumption~\ref{ass:iid}, the process
$\{\Delta_{i}\}_{i\in\N_0}$ is i.i.d. with geometric  distribution
  \begin{equation}
    \label{eq:22}
    \Prob \{\Delta_{i}=j\}=(1-p_{0})p_{0}^{j-1},\qquad j\in\N,
  \end{equation}
see \cite{xx09}.
In terms of the quantities introduced above, the state of the
nominal plant~\eqref{eq:15} when Algorithm A$_1$ is used satisfies:
\begin{equation}
  \label{eq:25a}
  x(k_{i}+\ell)=
  \begin{cases}
    f^{\ell}(x(k_i)), &\\\quad\text{if  $\ell \in
      \big\{0,1,\dots,\min (N(k_i),\Delta_i)\big\}$,}\\ 
    {f}_{\textsc{ol}}^{\ell-N(k_i)}\big(f^{N(k_i)}(x(k_i))\big), 
     &\\\quad \text{if  $N(k_i)<\Delta_i$ and
      $\ell \in\{N(k_i)+1,\dots,\Delta_i \}$,}
  \end{cases}
\end{equation}
for all $k_i\in\mathcal{K}$. It is worth emphasizing that~\eqref{eq:25a}
describes the plant state trajectory $\{x(k)\}$ for all $k\in\N_0$.

\par By setting $\ell = \Delta_i$ in~(\ref{eq:25a}), we obtain that the state
in~\eqref{eq:15} when Algorithm A$_1$ is employed can be described \emph{at the
  instants} $k_i\in\mathcal{K}$ via:
\begin{multline}
  \label{eq:25}
  \Prob\big\{x(k_{i+1}) = \chi^+ \,\big|\, x(k_i)=\chi\big\}\\=
  \begin{cases}
   \Prob\{ \Delta_i \leq N(k_i)\}, &\text{if $\chi^+= f^{\Delta_i}(\chi)$,}\\
   1-\Prob\{ \Delta_i \leq N(k_i)\}, &\text{if $\chi^+=
     {f}_{\textsc{ol}}^{\Delta_i-N(k_i)}\big(f^{N(k_i)}(\chi)\big)$}, 
  \end{cases}
\end{multline}
where $\Delta_i\in\N$.

\par It is worth noting that in~(\ref{eq:25}), the number of possible values for
$x(k_{i+1})$ given $x(k_i)$ is countably infinite, whereas if the baseline algorithm is
used,  there are only two possibilities, see~(\ref{eq:10}). The terms $\Prob\{
\Delta_i \leq N(k_i)\}$  can be easily evaluated as per
the following lemma:
\begin{lem}
\label{lem:N_A1}
Suppose that  Assumption~\ref{ass:iid} holds, then
\begin{equation}
  \label{eq:23}
  \Prob\{ \Delta_i \leq N(k_i)\}=\frac{1}{1-p_0}\sum_{l=1}^\Lambda
  p_l (1-p_0^l),\quad \forall k_i\in\mathcal{K}. 
\end{equation}
\end{lem}
\begin{proof}
By~(\ref{eq:26}), the random variables $\Delta_i$ and
$N(k_i)$ are independent. Furthermore, the two processes $\{\Delta_i\}_{i\in\N_0}$ and $\{N\}_{\N_0}$
are i.i.d. Thus, we can condition upon $N(k_i)\geq 1$ to obtain:
\begin{equation*}
  \begin{split}
   & \Prob\{ \Delta_i \leq N(k_i)\} \\&= \Prob\{   \Delta_i \leq N(k) \,|\, k\in\mathcal{K}\}\\
     &=\sum_{l=0}^\Lambda \Prob\{ N(k) = l \,|\, k\in\mathcal{K}\}\\&\qquad\qquad\qquad \cdot
\Prob\{ \Delta_i \leq N(k)\,|\,N(k)=l, k\in\mathcal{K}\}\\
      &=\frac{1}{1-p_0}\sum_{l=1}^\Lambda p_l\cdot \Prob\{ \Delta_i \leq
      l\}
      = \sum_{l=1}^\Lambda p_l
      \sum_{j=1}^{l} p_0^{j-1}.
    \end{split}
\end{equation*}
\vspace{-3mm}
\end{proof}

\subsection{Main Results}
\label{sec:main-results}
As a consequence of~(\ref{eq:25}) and~(\ref{eq:23}), and since $u(k_i)$ is
determined by $x(k_i)$, if Algorithm A$_1$ is used, then the plant
state $\{x(k_i)\}$, with $k_i\in
{\mathcal{K}}$, is Markovian. Stability of the closed loop can  be
analyzed by using a stochastic Lyapunov function approach which, to some extent,
parallels that used to prove Theorem~\ref{thm:baseline}. To state our result, we
first give the following lemma:
\begin{lem}
  \label{lemma:anytime_inter}
  Consider~(\ref{eq:25}) and suppose that Assumptions~\ref{ass:iid}
  and~\ref{ass:bound_prob} hold. Then $ \forall \chi \in\R^n$ and  $\forall k_i,k_{i+1}\in\mathcal{K},$ we have
  \begin{equation}
    \label{eq:6}
    \E\big\{V(x({k_{i+1}}))\,\big|\,x(k_{i})=\chi\big\}
    \leq\Bigg(\sum_{l=1}^{\Lambda}p_{l}\Omega_{l}\Bigg)V(\chi), 
  \end{equation}
  where
  \begin{equation}
    \label{eq:8}
    \begin{split}
      \Omega_{l}&\eq \rho\frac{1-(p_{0}\rho)^{l}}{1-p_{0}\rho}
      +\alpha\frac{(p_{0}\rho)^{l}}{1-p_{0}\alpha}\in \R_{\geq 0},\quad
      l\in\{1,2,\dots,\Lambda\}. 
    \end{split}
\end{equation}
\end{lem}

\begin{proof}
We use the total probability formula
  twice. First, we condition on the length of the tentative control sequence
  calculated during $t\in(k_{i}T_s,(k_i+1)T_s)$:
  \begin{equation}
    \label{eq:cond_on_N}
    \begin{split}
      &\E\{V(x({k_{i+1}}))\,|\,x(k_{i})\}\\&
      =\E\big\{\E\big\{V(x({k_{i+1}}))\,\big|\,x(k_{i}),N(k_{i})\big\}\big\}\\ 
      &=\sum_{l=1}^{\Lambda}\E\big\{V(x({k_{i+1}}))\,\big|\,x(k_{i}),N(k_{i})=l\big\}\\&\qquad\qquad\qquad\cdot
      \Prob\{ N(k) = l \,|\,k\in\mathcal{K}\}\\ 
      &=\sum_{l=1}^{\Lambda}\frac{p_{l}}{{1-p_0}}
      \E\big\{V(x({k_{i+1}}))\,\big|\,x(k_{i}),N(k_{i})=l  \big\}.
    \end{split}
  \end{equation}

 We note that in Algorithm $\textsc{A}_{1}$ previously calculated control values are
  erased at the instant $k_{i}$ and, thus,~(\ref{eq:25}) holds. Consequently,
  the conditional expectation
  $\E\{V(x({k_{i+1}}))\,|\,x(k_{i}),N(k_{i})\}$ can
  be evaluated by conditioning further on $\Delta_{i}$:
  \begin{equation}
    \label{eq:23q}
    \begin{split}
    &\E\big\{V(x({k_{i+1}}))\,\big|\,x(k_{i}),N(k_{i})\big\}\\&
    =\E\big\{\E\big\{ V(x({k_{i+1}}))\,\big|\,x(k_{i}),N(k_{i}),\Delta_{i}\big\}\big\}\\
    &=\sum_{j=1}^{\infty}(1-p_{0})p_{0}^{j-1}
    \E\big\{V(x({k_{i+1}}))\,\big|\,x(k_{i}),N(k_{i}),\Delta_{i}=j\big\},  
  \end{split}
\end{equation}
where we have
  used~(\ref{eq:22}).
  Now, using Assumption~\ref{ass:bound_prob} and Equation~(\ref{eq:14}), we
  obtain the bound:
  \begin{displaymath}
  \begin{split}
    &\E\big\{V(x({k_{i+1}}))\,\big|\,x(k_{i})=\chi,N(k_{i})=l,\Delta_{i}=j\big\}
    \\&\qquad\qquad\leq
    \begin{cases}
      \rho^{j}V(\chi),&\text{if $j\leq l$,}\\
      \alpha^{j-l}\rho^{l}V(\chi),&\text{if $j>l$.}
    \end{cases}
    \end{split}
  \end{displaymath}
  Thus,~(\ref{eq:23q}) gives:
  \begin{equation*}
    \begin{split}
     & \E\big\{V(x({k_{i+1}}))\,\big|\,x(k_{i})=\chi,N(k_{i})=l\big\} \\
      &\leq
      (1-p_{0})\left(\sum_{j=1}^{l}p_{0}^{j-1}\rho^{j}
        +\sum_{j=l+1}^{\infty}p_{0}^{j-1}\alpha^{j-l}\rho^{l}\right) V(\chi)\\
      &= (1-p_0)\Omega_{l}V(\chi),
    \end{split}
\end{equation*}
since, by Assumption~\ref{ass:bound_prob}, we have $p_0\alpha<1$.
 Substitution into~(\ref{eq:cond_on_N})
establishes~\eqref{eq:6}.
\end{proof}

Despite the fact that Lemma~\ref{lemma:anytime_inter} considers only the time
instants $k\in\N_0$ where $N(k)>0$, see~\eqref{eq:21a}, the bound in~\eqref{eq:6}
can be used to conclude about stochastic 
stability  (for all
$k\in\N_0$). 

\begin{thm}
  \label{theorem:a1_stability}
  Suppose that
  Assumptions~\ref{ass:CLF}--\ref{ass:bound_prob} hold and define
 \begin{equation}
   \label{eq:86}
   \sigma \eq \frac{1}{1-p_0\rho}\bigg(\rho (1-p_0\alpha) 
     +\frac{\alpha-\rho}{1-p_0}\sum_{l=1}^{\Lambda}p_{l}(p_{0}\rho)^{l}\bigg)
     \in \R_{\geq 0}. 
 \end{equation}
     

If
  \begin{equation}
  \label{eq:85}
  p_0\alpha + (1-p_0)\sigma  <1,
\end{equation}
 then the system~(\ref{eq:25a}) (with
  state trajectory $\{x\}_{\N_0}$) is stochastically
  stable.
\end{thm}

\begin{proof}
By Lemma~\ref{lemma:anytime_inter} and since
$\{x\}_{\mathcal{K}}$ is Markovian, we have that if 
\begin{equation}
    \label{eq:62}
    \Omega\triangleq\sum_{l=1}^{\Lambda}p_{l}\Omega_{l}<1,
  \end{equation}
  where $\Omega_{l}$ are  defined in~(\ref{eq:8}), then
 \begin{equation*}
    \E\big\{V(x({k_{i+1}}))\,\big|\,x(k_{i}),x(k_{i-1}),\dots, x(k_0)\big\}
    \leq\Omega V(x(k_{i})).
  \end{equation*}
Since, by Assumption~\ref{ass:CLF}, $V\colon \R^n\to\R_{\geq 0}$, we conclude that $V$ is a
  stochastic Lyapunov function for~(\ref{eq:25}); c.f., \cite{k71,meyn89}.
 
\par Direct calculations yield that 
\begin{equation}
  \label{eq:80}
  \begin{split}
    \Omega&=\sum_{l=1}^{\Lambda}p_{l}\Omega_{l}%
  %
=
      \frac{\rho}{1-p_0\rho}\sum_{l=1}^{\Lambda}p_{l} 
      + \sum_{l=1}^{\Lambda}p_{l}\frac{(p_{0}\rho)^{l} (\alpha- \rho) 
      }{(1-p_{0}\rho)(1-p_{0}\alpha)}\\
      &=
      \frac{\rho(1-p_0)}{1-p_0\rho}
      + \frac{ (\alpha- \rho) 
      }{(1-p_{0}\rho)(1-p_{0}\alpha)}\sum_{l=1}^{\Lambda}p_{l}(p_{0}\rho)^{l}.
  \end{split}
\end{equation}
Using equations~\eqref{eq:80} and~\eqref{eq:86}, we obtain
\begin{equation*}
  \begin{split}
    \Omega &=
      \frac{\rho(1-p_0)}{1-p_0\rho}
      + \frac{ (1-p_0)\big(\sigma (1-p_0\rho) - \rho (1-p_0\alpha) \big)
      }{(1-p_{0}\rho)(1-p_{0}\alpha)}\\
      &=\frac{(1-p_0)\big(  \sigma - p_0\rho\sigma
        \big)}{(1-p_{0}\rho)(1-p_{0}\alpha)}  
      =\frac{(1-p_0)  \sigma }{(1-p_{0}\alpha)} .
  \end{split} 
\end{equation*}
Hence,~\eqref{eq:85} is equivalent to~\eqref{eq:62}.
As a consequence, if~\eqref{eq:85} holds, then~\cite[Chapter
  8.4.2, Theorem 2]{k71} implies exponential stability at the instants
  $k_{i}\in\mathcal{K}$, i.e., we have:
  \begin{equation}
    \label{eq:analysis_markov}
    \E\{V(x({k_{i}}))\,|\,x(k_{0})=\chi_0\}\leq\Omega^{i}V(\chi_0),
    \quad \forall
    i\in\N, \quad \forall \chi_0\in\R^n
  \end{equation}
  Now for the time instants $k\in\N \,\backslash\, \mathcal{K}$, i.e., at those
  time steps when 
  no control input is calculated, we proceed as  in
  the proof of Lemma~\ref{lemma:anytime_inter}, to obtain that
  \begin{equation*}
    \begin{split}
    &\E\left\{\sum_{k=k_{i}}^{k_{i+1}-1}V(x({k}))\,\bigg|\,x(k_{i})\right\}\\
    &=\sum_{l=1}^{\Lambda} 
    \frac{ p_{l}}{1-p_0}
    \E\left\{\sum_{k=k_{i}}^{k_{i+1}-1}V(x({k}))\,\bigg|\,x(k_{i}),N(k_i)=l\right\}\\
    & =\sum_{l=1}^{\Lambda}p_{l}\sum_{j=1}^{\infty}p_{0}^{j-1}
    \E\left\{\sum_{k=k_{i}}^{k_{i+1}-1}V(x({k}))\,\bigg|\,
      x(k_{i}),N(k_i)=l,\Delta_{i}=j\right\}.  
  \end{split}
\end{equation*}
Since $\rho<1<\alpha$, we can bound
  \begin{multline*}
    \E\Bigg\{\!\sum_{k=k_{i}}^{k_{i+1}-1}\!\!V(x({k}))\,\bigg|\,
    x(k_{i}),N(k_{i})=l,\Delta_{i}=j\Bigg\}
   \\\leq \sum_{k=0}^{j-1}\alpha^{k}\E\big\{V(x(k_{i}))\,|\,x(k_{i})=\chi\big\}
   =\frac{\alpha^{j}-1}{\alpha-1} V(\chi) 
 \end{multline*}
so that
  \begin{multline*}
     \E\left\{\sum_{k=k_{i}}^{k_{i+1}-1}V(x({k}))\,\bigg|\,x(k_{i})=\chi,N(k_{i})\right\}\\\leq
    \frac{(1-p_{0})}{\alpha-1}\sum_{j=1}^{\infty}(\alpha^{j}-1)p_{0}^{j-1}V(\chi)
    =\frac{1}{1-p_{0}\alpha}V(\chi), 
  \end{multline*}
  in turn yielding
  \begin{multline}
    \label{eq:63}
    \E\left\{\sum_{k=k_{i}}^{k_{i+1}-1}V(x({k}))\,\bigg|\, x(k_{i})=\chi\right\}
    \\\leq \sum_{l=1}^{\Lambda}\frac{p_{l}}{(1-p_0)(1-p_{0}\alpha)}V(\chi)
    =\beta V(\chi),
  \end{multline}
where $ \beta\eq {1}/({1-p_{0}\alpha})\in\R_{\geq 0}$.
The expectation on the left hand side of~(\ref{eq:63}) is taken with respect to
  the distributions of $N(k_i)$ and $\Delta_i$. Since $\{x\}_{\K}$
  is Markovian and $N(k_i)$ and $\Delta_i$ are independent, we can take
  conditional expectation $\E\{\, \cdot\,|\, x(k_0)\}$ on both sides
  of~(\ref{eq:63}) to obtain:
\begin{align*}
&\E\Bigg\{ \E\Bigg\{\sum_{k=k_{i}}^{k_{i+1}-1}V(x({k}))\,\bigg|\, x(k_{i})\Bigg\}
    \,\Bigg|\, x(k_0)=\chi_0\Bigg\}\\&\qquad\leq\beta
  \,\E\big\{V(x({k_{i}}))\,|\,x(k_{0})=\chi_0\big\}\\
  \Rightarrow \: &\E\Bigg\{
    \E\bigg\{\sum_{k=k_{i}}^{k_{i+1}-1}V(x({k}))\,\bigg|\,
      x(k_{i}),x(k_{i-1}),\dots,x(k_0)\bigg\}\,\\&\qquad\qquad\qquad\qquad \bigg|\, x(k_0)=\chi_0\Bigg\} \leq \beta\, \Omega^i V(\chi_0)\\
   \Rightarrow \; &     \E\left\{\sum_{k=k_{i}}^{k_{i+1}-1}V(x({k}))\,\bigg|\, x(k_0)=\chi_0\right\}\\&\qquad\qquad\leq \beta
  \,\E\big\{V(x({k_{i}}))\,|\,x(k_{0})=\chi_0\big\} \leq \beta\, \Omega^i V(\chi_0),
    \end{align*}
where we have used the bound
in~(\ref{eq:analysis_markov}). Since we assume that $k_0=0$, 
this gives
  \begin{displaymath}
  \begin{split}
&    \E\left\{\sum_{k={0}}^{k_{j+1}-1}V(x({k}))\,\bigg|\,x({0})=\chi_0\right\}
    \leq 
    \beta
    \sum_{i=0}^{j}\Omega^{i}V(\chi_0) \\&\qquad\qquad= \beta \frac{1-\Omega^{j+1}}{1-\Omega}V(\chi_0)
    \leq \frac{\beta}{1-\Omega} V(\chi_0).
    \end{split}
  \end{displaymath}
   Thus, by letting 
  $k_{j+1}\to\infty$, it follows that there exists
  $c<\infty$ such that $$\sum_{k=0}^{\infty} \E\big\{V(k)\,\big|\,x(0)=\chi_0\big\}\leq c
  V(\chi_0).$$ The remainder of the proof now follows as in the proof of
  Theorem~\ref{thm:baseline}. 
\end{proof}

Theorem~\ref{theorem:a1_stability} establishes sufficient conditions for
stochastic stability of the closed loop when processor availability is i.i.d.\
and Algorithm A$_1$ is used. The quantity introduced in~(\ref{eq:86}) involves
the 
distribution of $\{N\}_{\N_0}$, the contraction factor of the baseline controller
$\kappa$,
see~(\ref{eq:3}), and the bound on 
the rate of increase of the plant state when left in open loop,
see~(\ref{eq:20}).

\par As a particular case, suppose that the distribution of $\{N\}_{\N_0}$
satisfies $p_1=1-p_0$, i.e., the processor time availability is such that the
Algorithm A$_1$  provides at most one control input.  In this case, expression~\eqref{eq:86} gives that
$\sigma=\rho$ 
  and, not surprisingly, we recover the sufficient condition for stochastic
  stability established for the baseline
algorithm~\eqref{eq:4} in~\eqref{eq:baseline_stability}. 

\par More generally, if the probability that Algorithm A$_1$
provides more than one control 
  value is non-zero, then 
 Theorem~\ref{theorem:a1_stability} establishes stochastic
stability for  a larger class of plant models  than
Theorem~\ref{thm:baseline}. This observation follows upon noting that
$\sigma$ can be rewritten as: 
\begin{equation*}
  \sigma = \rho -\frac{(\alpha-\rho)
  }{(1-p_0)(1-p_0\rho)}\sum_{l=1}^{\Lambda}p_{l} \big( \rho p_0-(\rho p_{0})^{l}\big).
\end{equation*}
Thus, if $p_{l^\star}>0$ for some  $l^\star\in\{2,3,\dots,\Lambda\}$, then
$\sum_{l=1}^{\Lambda}p_{l} \big( \rho p_0-(\rho 
p_{0})^{l}\big)>0$ and $\sigma <\rho$. 
This  suggests that Algorithm A$_1$
has better stabilizing properties than the baseline algorithm. 
 


\section{Stability with Algorithm A$_2$}
\label{sec:algorithm-a_2}
We first note that for $\Lambda \in \{1,2\}$, Algorithm A$_2$ is equivalent to
Algorithm A$_1$. Henceforth, we focus on cases where $\Lambda > 2$. 
It follows directly from~(\ref{eq:18}) and~(\ref{eq:9}) that with Algorithm
A$_2$ if
$\Delta_i>N(k_i)$ and $\lambda(k_i-1)>N(k_i)+1$, then  
$\lambda(k_i)=\lambda(k_i-1)-1>N(k_i)$ and the
plant input at  times 
$\{k_i+N(k_i),k_i  +N(k_i)+ 1, \dots , k_i+\min (\lambda(k_i),\Delta_i) - 1\}$
will stem from buffer contents at
 time $k_i-1$, see also Example~\ref{ex:one}. Thus, with Algorithm
$\textsc{A}_{2}$, 
$\{x\}_{\mathcal{K}}$ and $\{x\}_{\N_0}$ are not Markovian and the analysis
carried out for
Algorithm $\textsc{A}_{1}$ does not carry over directly. 

\par To recover a Markovian structure, consider the overall system
state $\{\theta\}_{\N_0}$ defined via: 
\begin{equation}
\label{eq:58b}
  \theta(k) \eq
  \begin{bmatrix}
    x(k)\\b(k-1)
  \end{bmatrix}.
\end{equation}
In terms of $\theta(k)$,~\eqref{eq:18} gives that at all times where
$N(k)=0$, the plant input is given by  
\begin{equation}
  \label{eq:28}
  u(k) = 
  \begin{bmatrix}
    0_{p\times (n+p)} & I_p & 0_{p \times (\Lambda-2)p}
  \end{bmatrix}
  \theta(k),\quad \Lambda>2.
\end{equation}
Furthermore, $\{\theta\}_{\N_0}$ and thereby also
$\{\theta\}_{\mathcal{K}}$ are Markovian processes.
The mapping\footnote{For example, for $j=1$ we have 
$ {f}_{\textsc{b}}^1(\theta(k))=
\begin{bmatrix} f(x(k),b_2(k-1),\mathbf{0}_m)\\
  S b(k-1)
\end{bmatrix}
$, see also~\eqref{eq:12}.}
\begin{equation}
  \label{eq:29}
  \begin{split}
&  {f}_{\textsc{b}}^j(\theta)\\&\eq
  \begin{cases}
   \theta, & \text{if $j=0$},\\
   \begin{bmatrix}
      {f}\Big(
M_{1} {f}_{\textsc{b}}^{j-1}(\theta),
     M_{2}
{f}_{\textsc{b}}^{j-1}(\theta),\mathbf{0}_m \Big)\\
S^j  \Big[\begin{matrix} 0_{\Lambda p \times n}&I_{\Lambda p}
\end{matrix}\Big] \theta
\end{bmatrix}
,&\text{if $j \geq 1$}
  \end{cases}
  \end{split}
\end{equation}
where
\begin{align*}
M_{1}&=    \Big[\begin{matrix}
      I_n&0_{n\times \Lambda p}
    \end{matrix} \Big]\\
    M_{2}&=\Big[\begin{matrix}
    0_{p\times (n+p)} & I_p & 0_{p \times (\Lambda-2)p}
  \end{matrix} \Big],
  \end{align*}
 allows one  to characterize the nominal system behaviour at times where
computational resources are insufficient to calculate control inputs, so that buffered
plant inputs are used. More precisely, the  nominal plant state when
Algorithm A$_2$ is used can be stated in terms of a random mapping
with inputs $\{\theta\}_{\mathcal{K}}$ as follows:
\begin{equation}
  \label{eq:30}
\begin{split} & x(k_{i}+\ell)\\&=
  \begin{cases}
    f^{\ell}(x(k_i)), &\text{if  $\ell \in
      \big\{0,1,\dots,\min (N(k_i),\Delta_i)\big\}$}\\
    M_{1}
    {f}_{\textsc{b}}^{\ell-N(k_i)}(\theta'), &\text{if
       $\Delta_i>N(k_i)$ and $\lambda(k_i)>N(k_i)$}\\
    & \text{and
      $\ell \in\{N(k_i)+1,\dots,$}\\&\qquad\qquad\text{$\min (\lambda(k_i),\Delta_i) \}$}\\
    {f}_{\textsc{ol}}^{\ell-\lambda(k_i)}(x'), 
    &\text{if  $\Delta_i>\lambda(k_i)$ }\\&\text{and
      $\ell \in\{\lambda(k_i)+1,\dots,\Delta_i \}$,}
  \end{cases}
  \end{split}
\end{equation}
 where  $f^{\ell}$ and
${f}_{\textsc{ol}}^{j}$ are defined in~\eqref{eq:24} and~\eqref{eq:27},
respectively,  $\lambda(k_i)= \max\{N(k_i),
\lambda(k_i-1)-1\}$, and with
\begin{equation}
  \label{eq:31}
  \begin{split}
    \theta'&\eq
    \begin{bmatrix}
      f^{N(k_i)}(x(k_i))\\ S^{N(k_i)} b(k_i-1)
    \end{bmatrix}
    \\
    x'&\eq  \begin{bmatrix}
      I_n&0_{n\times \Lambda p}
    \end{bmatrix}
    f_{\textsc{b}}^{\lambda(k_i)- N(k_i)}(\theta').
  \end{split}
\end{equation}
  At the
  instants $k_i\in\mathcal{K}$, the nominal plant state
in~\eqref{eq:15} when Algorithm A$_2$ is used can thus be described  via: 
\begin{equation}
  \label{eq:32}
  \begin{split}
    &\Prob\big\{x(k_{i+1}) = \chi^+ \,\big|\,
    x(k_i)=\chi,b(k_i-1)=\upsilon\big\}\\
    &=
    \begin{cases}
      \Prob\{ \Delta_i \leq N(k_i)\}, &\text{if $\chi^+= f^{\Delta_i}(\chi)$,}\\
      \Prob\{N(k_i)< \Delta_i \leq \lambda (k_i)\}, &\text{if $\chi^+=
       M_{1}
        {f}_{\textsc{b}}^{\Delta_i-N(k_i)}(\vartheta)$}, \\
      \Prob\{\Delta_i > \lambda (k_i)\}, &\text{if $\chi^+=
        {f}_{\textsc{ol}}^{\Delta_i-N(k_i)}(\xi)$},
    \end{cases}
  \end{split}
\end{equation}
where
\begin{equation}
  \label{eq:35}
    \vartheta=
    \begin{bmatrix}
      f^{N(k_i)}(\chi)\\ S^{N(k_i)} \upsilon
    \end{bmatrix}
    ,\quad
    \xi=  \begin{bmatrix}
      I_n&0_{n\times \Lambda p}
    \end{bmatrix}
    f_{\textsc{b}}^{\lambda(k_i)- N(k_i)}(\vartheta).
 \end{equation}
 Note that, as shown in Lemma~\ref{lem:N_A1}, the probabilities  $\Prob\{
\Delta_i \leq N(k_i)\}$ used in~\eqref{eq:32}  are i.i.d. Nevertheless, it is easy to see that
  \begin{equation*}
    \label{eq:48}
    \begin{split}
     & \Prob\{N(k_i)< \Delta_i  \leq \lambda (k_i)\}\\&
      = \frac{1}{1-p_0}\sum_{l=1}^\Lambda p_l\cdot \Prob \big\{l< \Delta_i  \leq 
      \max\{l, \lambda(k_i-1)-1\}\big\}\\
      &= \sum_{l=1}^\Lambda p_l\!\! \sum_{j=l+1}^{\max\{l, \lambda(k_i-1)-1\}}p_0^{j-1}\\
&      =\frac{1}{1-p_0}\sum_{l=1}^\Lambda p_l \big(p_0^l-p_0^{\max\{l,
        \lambda(k_i-1)-1\}}\big),
    \end{split}
  \end{equation*}
expression which depends upon $\lambda(k_i-1)$ and therefore on $b(k_i-1)$.


\par  The following stochastic stability result is akin to the one developed in
Section~\ref{sec:main-results} for Algorithm A$_1$. It shows that the sufficient
condition developed for Algorithm A$_1$ is also sufficient to guarantee
stochastic stability when Algorithm A$_2$ is used.

\begin{thm}
  \label{theorem:a2_stability}
   Suppose that Assumptions~\ref{ass:CLF} to \ref{ass:bound_prob} 
   hold and that Algorithm $\textsc{A}_{2}$ is  used.
   If~\eqref{eq:85} is satisfied, then
 the closed loop system (with
  state trajectory $\{x\}_{\N_0}$) is stochastically
  stable. \hfs
\end{thm}
\begin{proof}
It follows from~\eqref{eq:32},~\eqref{eq:29} and by proceeding as in
the proof of Lemma~\ref{lemma:anytime_inter} that
\begin{equation}
  \label{eq:36}
  \begin{split}
 &   \E\big\{V(x({k_{i+1}}))\,\big|\, \theta(k_i)\big\}
   \\& =
    \sum_{l=1}^{\Lambda}p_{l}
    \sum_{j=1}^{\infty}p_{0}^{j-1}\E\big\{V(x({k_{i+1}}))\,\big|\, 
    \theta(k_i), N(k_i)=l,\Delta_i=j\}.
  \end{split}
\end{equation}
On the other hand, since $\lambda(k_i)$ is a function of $N(k_i)$ and
$b(k_i-1)$, we have
\begin{equation*}
  \begin{split}
&    \E\big\{V(x({k_{i+1}}))\,\big|\, x(k_i)=\chi,b(k_i-1)=\upsilon,
    N(k_i)=l,\Delta_i=j\} \\ 
    &=   \E\big\{V(x({k_{i+1}}))\,\big|\, \\&\qquad\quad x(k_i)=\chi,b(k_i-1)=\upsilon,
    N(k_i)=l,\Delta_i=j,\lambda(k_i)=\lambda\}\\
&\leq \begin{cases}
      \rho^{j}V(\chi),&\text{if $j\leq \lambda$,}\\
      \alpha^{j-\lambda}\rho^{\lambda}V(\chi),&\text{if $j>\lambda$,}
    \end{cases}
  \end{split}
\end{equation*}
where we have used the bounds in~\eqref{eq:14},~\eqref{eq:20} and where
$\lambda= \max\{l,\lambda_0-1\}$ with $\lambda_0$ denoting the 
index of the last nonzero entry in $\upsilon$, see~\eqref{eq:9}. Substitution
into~\eqref{eq:36} yields that
\begin{equation*}
  \begin{split}
    & \E\big\{V(x({k_{i+1}}))\,\big|\, x(k_i)=\chi, b(k_i-1)=\upsilon \big\}
     \\&\leq \sum_{l=1}^{\Lambda}p_{l}
    \Bigg(\sum_{j=1}^{\lambda}p_{0}^{j-1}\rho^j +\sum_{j=\lambda+1}^\infty
    p_{0}^{j-1}\alpha^{j-\lambda}\rho^\lambda \Bigg) V(\chi)\\
    &\leq   \sum_{l=1}^{\Lambda}p_{l}
    \bigg(\sum_{j=1}^{l}p_{0}^{j-1}\rho^j +\sum_{j=l+1}^\infty
    p_{0}^{j-1}\alpha^{j-l}\rho^l \bigg) V(\chi)\\&=\Omega  V(\chi),\quad \forall
    (\chi,\upsilon)\in \R^n \times \R^{\Lambda p}
  \end{split}
\end{equation*}
where $ \Omega$ is defined in~\eqref{eq:62} and where we have used the fact that 
$\rho<1<\alpha$ and $\lambda\geq l$.
\par  Since $\{\theta\}_{\mathcal{K}}$ is  Markovian, it 
follows from\cite[Chapter  8.4.2, Theorem 2]{k71}  that 
\begin{multline*}
   \E\big\{V(x({k_{i}}))\,\big|\,x(k_{0})=\chi_0,b(k_0-1)=\upsilon_0 \big\}
 \\  \leq \Omega^{i}V(\chi_0),
   \quad 
   \forall (i,\chi_0,\upsilon_0)\in\N \times\R^n \times \R^{\Lambda p}.
 \end{multline*}
The remainder of the proof now follows, \emph{mutatis mutandis}, that of  Theorem~\ref{theorem:a1_stability}.
\end{proof}

\section{Markovian Processor State Model}
\label{sec:mark-chain-proc}
So far we have assumed that the process $\{N\}_{\N_0}$ is i.i.d. In  situations
where the control loop is shared with other applications having time-varying and correlated
processing demands it is likely that Assumption~\ref{ass:iid} will not be satisfied. We 
will next outline how the analysis presented can be extended to 
encompass cases where the processor availability for control, henceforth modeled via
the processor state process
$\{g\}_{\N_0}$,  is correlated.
\begin{ass}
  \label{ass:ge}
  The processor state process $\{g\}_{\N_0}$ is an  irreducible aperiodic finite Markov Chain (see,
  e.g.,\cite{lawler06}) with values in the finite set $\{1,2,\dots,G\}$,
  $G\in\N$. Its transition 
  matrix $Q$ is given by
  \begin{equation}
    \label{eq:49}
    Q=
    \begin{bmatrix}
      q_{11}&q_{12}&\dots&q_{1G}\\
       q_{21}&q_{22}& & q_{2G}\\
       \vdots & \vdots &\ddots&\vdots \\
       q_{G1} & q_{G2}&\dots &q_{GG}
    \end{bmatrix},
  \end{equation}
where $
  q_{ij} = \Prob\{ g(k+1) = j\,|\, g(k) = i\}$, $\forall i,j\in \{1,2,\dots,G\}.$
Given any processor state $g(k)=\varsigma$, $\forall
  (l,\varsigma)\in\{0,1,\dots,\Lambda\} \times \{1,2,\dots,G\},$ the conditional distribution of the
process $\{N\}_{\N_0}$ satisfies
\begin{equation}
  \label{eq:40}
  \Prob\{N(k)=l\,|\, g(k)=\varsigma\} = p_{l|\varsigma},
\end{equation}
with given probabilities $p_{l|\varsigma}$.\hfs
\end{ass}
For the baseline algorithm in~\eqref{eq:4} stochastic stability can be ensured
as follows:
\begin{thm}
  Suppose that Assumptions~\ref{ass:CLF},~\ref{ass:bound_prob} and~\ref{ass:ge}
  hold. Define
  \begin{equation}
    \label{eq:75}
    \hat{p}_{0}=\max_{\varsigma\in\{1,2,\dots,G\}}{p_{0|\varsigma}}.
  \end{equation}
A sufficient condition for~\eqref{eq:10} to be stochastically stable is that
  \begin{equation}
    \label{eq:73}
    \hat{p}_0\alpha +(1-\hat{p}_0)\rho<1.
  \end{equation}
\end{thm}
\begin{proof}
  First, we note that the joint process $\{(x,g)\}_{\N_0}$ is
  Markovian. Thus, by using the law of total expectation and the fact that $\alpha-\rho>0$, we obtain
  \begin{equation}
    \label{eq:74}
    \begin{split}
     &\E\big\{V(x(1))\,\big|\,x(0)=\chi,g(0)=\varsigma \big\}\\\leq&\quad
     (p_{0|\varsigma}\alpha +(1-p_{0|\varsigma})\rho)V(\chi)\\\leq&\quad (
     \hat{p}_0\alpha +(1-\hat{p}_0)\rho)V(\chi),
     \end{split}
   \end{equation}
 for all $(\chi,\varsigma) \in \R^n \times \{1,2,\dots,G\}$. The remainder of
 the proof now follows as in the proof of Theorem~\ref{thm:baseline} and is
 omitted for space constraints. 
\end{proof}

The stability results of
Sections~\ref{sec:algorithm-a_1} and~\ref{sec:algorithm-a_2} can be extended to
encompass the Markovian processor model of Assumption~\ref{ass:ge}. Here we only present the stability results for Algorithm A$_1$. The main difference from the analysis in Section~\ref{sec:algorithm-a_1} is the fact that the
plant state $\{x\}_{\mathcal{K}}$ is no longer
Markovian. Interestingly, the analysis can be extended by recognizing that the aggregated process $\{(x,g)\}_{\mathcal{K}}$ is 
Markovian. 

\par Whilst the process $\{\Delta_{i}\}_{i\in\mathbb{N}_{0}}$ is no longer
i.i.d.,  the conditional distributions $\Prob\{\Delta_i  \,|\, 
g(k_i)\}$ can be evaluated as per the following result:
\begin{lem}
  \label{lem:delta_g}
 Suppose that Assumption~\ref{ass:ge} holds and define
 \begin{equation*}
   \G \eq \big\{\varsigma \in\{1,2,\dots, G\}\colon p_{0|\varsigma}<1\big\}.
 \end{equation*}
 Then\footnote{Note that in the
   i.i.d.\ case of Assumption~\ref{ass:iid}, we have $G=Q=1$, $\G=\{1\}$, $\bar Q=p_0$ and $\bar
p = 1-p_0$, so that~\eqref{eq:43} reduces to~\eqref{eq:22}.}
 \begin{equation}
   \label{eq:43}
   \Prob\{\Delta_i = j \,|\, g(k_i)=\varsigma\} = 
   \bar q_\varsigma
   \bar Q^{j-1} \bar p,\quad \forall (j,\varsigma) \in \N\times\G,
\end{equation}
where
\begin{equation}
  \label{eq:54}
  \begin{split}
    \bar q_\varsigma&\eq \begin{bmatrix}
     q_{\varsigma 1}  &\dots& q_{\varsigma G}\end{bmatrix} ,\\      \bar Q &\eq \diag (p_{0|1},\dots,p_{0|G}) Q,\\
    \bar p &\eq
    \begin{bmatrix}
      1-p_{0|1}&
      \dots&
      1-p_{0|G}
    \end{bmatrix}^T.
  \end{split}
\end{equation}

\end{lem}
\begin{proof}
  Denote   
  \begin{equation*}
\begin{split}
      \vec{g}&=
      \begin{bmatrix}
        g(k_i+1) & g(k_i+2) &\dots& g(k_i+j)
      \end{bmatrix},\\
      \vec{\varsigma}&=
      \begin{bmatrix}
        \varsigma_1 & \varsigma_2 &\dots& \varsigma_j
      \end{bmatrix}
      \end{split}
  \end{equation*}
and $\mathcal{G}=\G^j$.
Conditioning upon the processor state sequence $\vec{g}\in\mathcal{G}$ gives that  
 \begin{equation*}
   \begin{split}
    &\Prob\{\Delta_i = j \,|\, g(k_i)=\varsigma\}\\&= \sum_{\vec{\varsigma}\in\mathcal{G}}
    \Prob\big\{\Delta_i = j \,|\, g(k_i)=\varsigma,\vec{g}=\vec{\varsigma}\big\}
    \,\Prob\{\vec{g}=\vec{\varsigma}\,|\, g(k_i)=\varsigma\}\\
    &=\sum_{\vec{\varsigma}\in\mathcal{G}}
    p_{0|\varsigma_1}p_{0|\varsigma_2}
    \dots p_{0|\varsigma_{j-1}} (1-p_{0|\varsigma_{j}}) q_{\varsigma \varsigma_1}
 q_{\varsigma_1 \varsigma_2}
 \dots  q_{\varsigma_{j-1} \varsigma_j}\\
 &= \sum_{\vec{\varsigma}\in\mathcal{G}}
( q_{\varsigma \varsigma_1}p_{0|\varsigma_1})
( q_{\varsigma_1 \varsigma_2}p_{0|\varsigma_2}) \dots  
 ( q_{\varsigma_{j-2} \varsigma_{j-1}} p_{0|\varsigma_{j-1}})\\&\qquad\qquad\qquad\qquad\qquad\qquad\qquad\qquad
( q_{\varsigma_{j-1} \varsigma_j} (1-p_{0|\varsigma_{j}})),
   \end{split}
 \end{equation*}
which can be rewritten in compact form as in~\eqref{eq:43}.
\end{proof}

\par The state evolution at times $k_{i}\in\mathcal{K}$ can now be evaluated as 
\begin{equation}
  \label{eq:55}
  \begin{split}
&  \Prob\big\{x(k_{i+1}) = \chi^+ \,\big|\, x(k_i)=\chi,g(k_i)=\varsigma \big\}\\&=
  \begin{cases}
   \Prob\{ \Delta_i \leq N(k_i)\,|\,g(k_i)=\varsigma\},\\
   \qquad\qquad \text{if $\chi^+=
     f^{\Delta_i}(\chi)$,}\\ 
   1-\Prob\{ \Delta_i \leq N(k_i)\,|\,g(k_i)=\varsigma\}, \\
   \qquad\qquad \text{if $\chi^+=
     {f}_{\textsc{ol}}^{\Delta_i-N(k_i)}\big(f^{N(k_i)}(\chi)\big)$}, 
  \end{cases}
  \end{split}
\end{equation}
where $\Delta_i\in\N$ and $\varsigma \in \G$. 
Lemma~\ref{lemma:anytime_inter} can be generalized as follows:
\begin{lem}
\label{lem:A1Markov}
  Consider~\eqref{eq:55} and suppose that Assumption~\ref{ass:bound_prob}
  with $p_0$ replaced by $\hat{p}_{0}$, 
  and Assumption~\ref{ass:ge} hold. Then
  \begin{equation*}
    \E\{V(x({k_{i+1}}))\,|\,x(k_{i})=\chi,g(k_i)=\varsigma\}\leq
    \Upsilon_\varsigma V(\chi)
  \end{equation*}
for all $(\chi,\varsigma) \in \R^n\times \G$, 
where
\begin{equation}
\label{eq:37}
  \Upsilon_\varsigma\eq \bar q_\varsigma\big(I-\rho\bar{Q}\big)^{-1} \bigg(
 \rho I + \frac{  (\alpha-\rho)}{1-p_{0|\varsigma}}\big(I-\alpha\bar{Q}\big)^{-1}\sum_{l=1}^\Lambda
   p_{l|\varsigma}(\rho\bar{Q})^l\bigg) \bar p,
\end{equation}
with $\bar q_\varsigma$, $\bar{Q}$, and $\bar p$ as in~(\ref{eq:54}).
\end{lem}
\begin{proof}
  Following as in the proof of Lemma~\ref{lemma:anytime_inter}, we first
  condition upon $N(k_i)$ to calculate, for $\varsigma \in\G$,
  \begin{multline}
    \label{eq:59}
      \E\{V(x({k_{i+1}}))\,|\,x(k_{i}),g(k_i)=\varsigma\}\\= \sum_{l=1}^\Lambda
      \frac{p_{l|\varsigma}}{1-p_{0|\varsigma}}
      \E\big\{V(x({k_{i+1}}))\,\big|\,x(k_{i}),N(k_{i})=l,g(k_i)=\varsigma  \big\} 
   \end{multline}
and then condition further on $\Delta_i$ to obtain that 
\begin{equation*}
  \begin{split}
    &\E\big\{V(x({k_{i+1}}))\,\big|\,x(k_{i}),N(k_{i}),g(k_i)=\varsigma \big\}=\\ 
    &
    \sum_{j=1}^{\infty} \bar q_\varsigma \bar Q^{j-1} \bar p
    \E\big\{V(x({k_{i+1}}))\,\big|\,x(k_{i}),N(k_{i}),
    \Delta_{i}=j,g(k_i)=\varsigma \big\},
  \end{split}
\end{equation*}
where we have used~\eqref{eq:43}. Equation~\eqref{eq:14} and
Assumption~\ref{ass:bound_prob} then provide the bound:
\begin{equation}
  \label{eq:61}
  \begin{split}
   & \E\big\{V(x({k_{i+1}}))\,\big|\,x(k_{i})=\chi,N(k_{i}),g(k_i)=\varsigma  \big\}\\
    &\leq
    \sum_{j=1}^{l} \bar q_\varsigma
   \bar Q^{j-1} \bar p
   \rho^j V(\chi)
   + \sum_{j=l+1}^{\infty} \bar q_\varsigma
   \bar Q^{j-1} \bar p
   \alpha^{j-l}\rho^l V(\chi)\\
   &= \bar q_\varsigma \bigg(  \sum_{j=1}^{l} 
   \bar Q^{j-1} 
   \rho^j    + \rho^l \sum_{j=l+1}^{\infty}
   \bar Q^{j-1} 
   \alpha^{j-l}\bigg) \bar pV(\chi),
  \end{split}
\end{equation}
Since  $Q$ is the transition probability of an
  irreducible aperiodic Markov Chain and $\rho \hat{p}_{0}\leq\alpha \hat{p}_{0}<1$, the above
  summation is convergent. If we now substitute~(\ref{eq:61})
  into~\eqref{eq:59}, then we obtain the bound
\begin{equation*}
  \begin{split}
   &\E\{V(x({k_{i+1}}))\,|\,x(k_{i})=\chi,g(k_i)=\varsigma\}\\&\leq
    \sum_{l=1}^\Lambda \frac{p_{l|\varsigma} \bar
      q_\varsigma}{1-p_{0|\varsigma}} \Bigg( \sum_{j=1}^{l} \bar Q^{j-1} \rho^j
    + \rho^l \sum_{j=l+1}^{\infty} \bar Q^{j-1} \alpha^{j-l}\Bigg) \bar p
    V(\chi)\\
  &=\frac{ \bar q_\varsigma}{1-p_{0|\varsigma}} \sum_{l=1}^\Lambda p_{l|\varsigma}
    \Big( \rho \big(I-(\rho\bar{Q})^l\big)\big(I-\rho\bar{Q}\big)^{-1} 
    \\&\qquad\qquad+\alpha (\rho\bar{Q})^l \big(I-\alpha\bar{Q}\big)^{-1} \Big) \bar p
    V(\chi)\\
    &=\frac{ \bar q_\varsigma}{1-p_{0|\varsigma}} \sum_{l=1}^\Lambda
    p_{l|\varsigma}\Big( \rho\big(I-\rho\bar{Q}\big)^{-1}\\&\qquad\qquad+ (\rho\bar{Q})^l
    \big(\alpha\big(I-\alpha\bar{Q}\big)^{-1} 
    -\rho\big(I-\rho\bar{Q}\big)^{-1}\big)\Big)\bar p
    V(\chi)
  \end{split}
\end{equation*}
where we have used\cite[Prop.\ 9.4.13]{bernst09}.
The result now follows upon noting that 
$\alpha\big(I-\alpha\bar{Q}\big)^{-1} 
    -\rho\big(I-\rho\bar{Q}\big)^{-1} 
    =\big(I-\rho\bar{Q}\big)^{-1} \big(\alpha\big(I-\rho\bar{Q}\big)
    -\rho \big(I-\alpha\bar{Q}\big)\big)\big(I-\alpha\bar{Q}\big)^{-1}
    =(\alpha-\rho)\big(I-\rho\bar{Q}\big)^{-1} \big(I-\alpha\bar{Q}\big)^{-1}$
and some algebraic manipulations.
\end{proof}
Following as in the proof of Theorem~\ref{theorem:a1_stability}, one can derive the following stochastic stability result:
\begin{thm}
  \label{thm:A1stabMarkov}
 Suppose that
  Assumption~\ref{ass:CLF} and the hypotheses of Lemma~\ref{lem:A1Markov}
  hold. If
  \begin{equation*}
    \Upsilon_{\varsigma} <1,\quad \forall \varsigma \in\G,
  \end{equation*}
   then the system~\eqref{eq:30} (with
  state trajectory $\{x\}_{\N_0}$) is stochastically
  stable.
\end{thm}
The above generalizes the analysis in Section~\ref{sec:algorithm-a_1} to
situations where the processor availability for control is correlated. The
results of Section~\ref{sec:algorithm-a_2} can  be similarly extended.

\begin{rem}
   It is easy to see that the i.i.d.\ model of Assumption~\ref{ass:iid}
  corresponds to the special case of the Markovian model in
  Assumption~\ref{ass:ge}, obtained by setting $G=1$, $\G=\{1\}$, $Q=q_{11}=1$ and $p_l=p_{l|1}$,
  for all $l\in\{0,1,\dots,\Lambda\}$. With the above parameters,~\eqref{eq:75}
  and~\eqref{eq:54} 
  give that $\hat{p}_0=p_0$, $ \bar q_\varsigma = q_{11}=1$, $ \bar Q =p_0$ and $\bar p = 1-p_0$.
Thus, the term $\Upsilon_\varsigma =\Upsilon_1$ in~\eqref{eq:37}
becomes
  \begin{equation*}
  \begin{split}
    \Upsilon_1&= \frac{1-p_0}{(1-p_0\rho)(1-p_0\alpha  \big)} \bigg(
 \rho(1-p_0\alpha  \big) + \frac{  (\alpha-\rho)}{1-p_{0}}\sum_{l=1}^\Lambda
   p_{l}(p_0\rho )^\ell\bigg)\\&=\frac{(1-p_0) \sigma}{1-p_0\alpha},
   \end{split}
 \end{equation*}
where $\sigma$ is given in~\eqref{eq:86}. Therefore, for the i.i.d.\ case,
$\Upsilon_1<1$ if and only 
if~\eqref{eq:85} holds, and Theorem~\ref{thm:A1stabMarkov} reduces to
Theorem~\ref{theorem:a1_stability}. \hfs
\end{rem}

\section{Numerical examples}
\label{sec:numerical_examples}
Having established sufficient conditions for stochastic stability of the anytime
control loops, we next study performance issues.  For that purpose, we assume
that the execution time available is i.i.d., uniformly distributed in the 
interval $[0,1]$. {{The}} execution time can also be viewed as the
fraction of  the
maximum possible processor time that is available at any time
step. Denote the time taken to calculate one control input by $\tau\in(0,1)$. The
probability distribution of $\{N\}_{\N_0}$, see~\eqref{eq:7}, is then given by 
\begin{equation}
\label{eq:78}
 p_l=\tau,\quad  \forall l\in\{0,1,\dots, \Lambda-1\},\quad p_\Lambda = 1 -
 \Lambda \cdot \tau,
\end{equation}
where $\Lambda=\lfloor {1}/{\tau}\rfloor$ is the maximum number of control
  inputs that can be calculated at any time step. Throughout this
section, tentative controls 
in~\eqref{eq:4b} are obtained by evaluating $\kappa$ for the 
corresponding predicted plant state.
\par   To evaluate control performance, we
consider the empirical  cost  
\begin{equation*}
   J=\frac{1}{10^{5}}\E\left\{\sum_{k=0}^{10^5-1}\left(0.2x^{2}(k)+2u^{2}(k)\right)\right\}, 
\end{equation*}
 where  expectation is taken with respect to the availability of execution time as
described above.  

\par We first consider a
nonlinear plant model (adapted from~\cite{ntk}):
\begin{equation}
\label{eq:2}
  x(k+1)=x(k)+0.01(x^{3}(k)+u(k)) + w(k),
\end{equation}
where $w(k)$ is white noise uniformly distributed in the interval
$[0,0.01]$.
The baseline control policy is taken as $\kappa(x)=-x^{3}-x$. It can be
verified that if one chooses $V(x)=|x|$, then Assumption~\ref{ass:CLF} is
satisfied for $\varphi_1(s)=\varphi_2(s)=s$  and $\rho=0.99$. 
 Fig.~\ref{fig:1} shows the percentage improvement in cost achieved as a
function of  the time taken to calculate one control input  for both algorithms A$_1$
and A$_2$, as compared to the baseline algorithm~\eqref{eq:4}. It can be
appreciated in that figure, both algorithms proposed give a significant
performance  improvement, with Algorithm A$_2$ further outperforming Algorithm
A$_1$.\footnote{A total of 1000 Monte Carlo simulations were used to generate
  the data. Of course, the obtained results are no more
than a case-by-case analysis, and consequently one cannot
conclude anything about the superiority of either algorithm in general.}
 



\begin{figure}[t]
  \centering
  \includegraphics[scale=0.48]{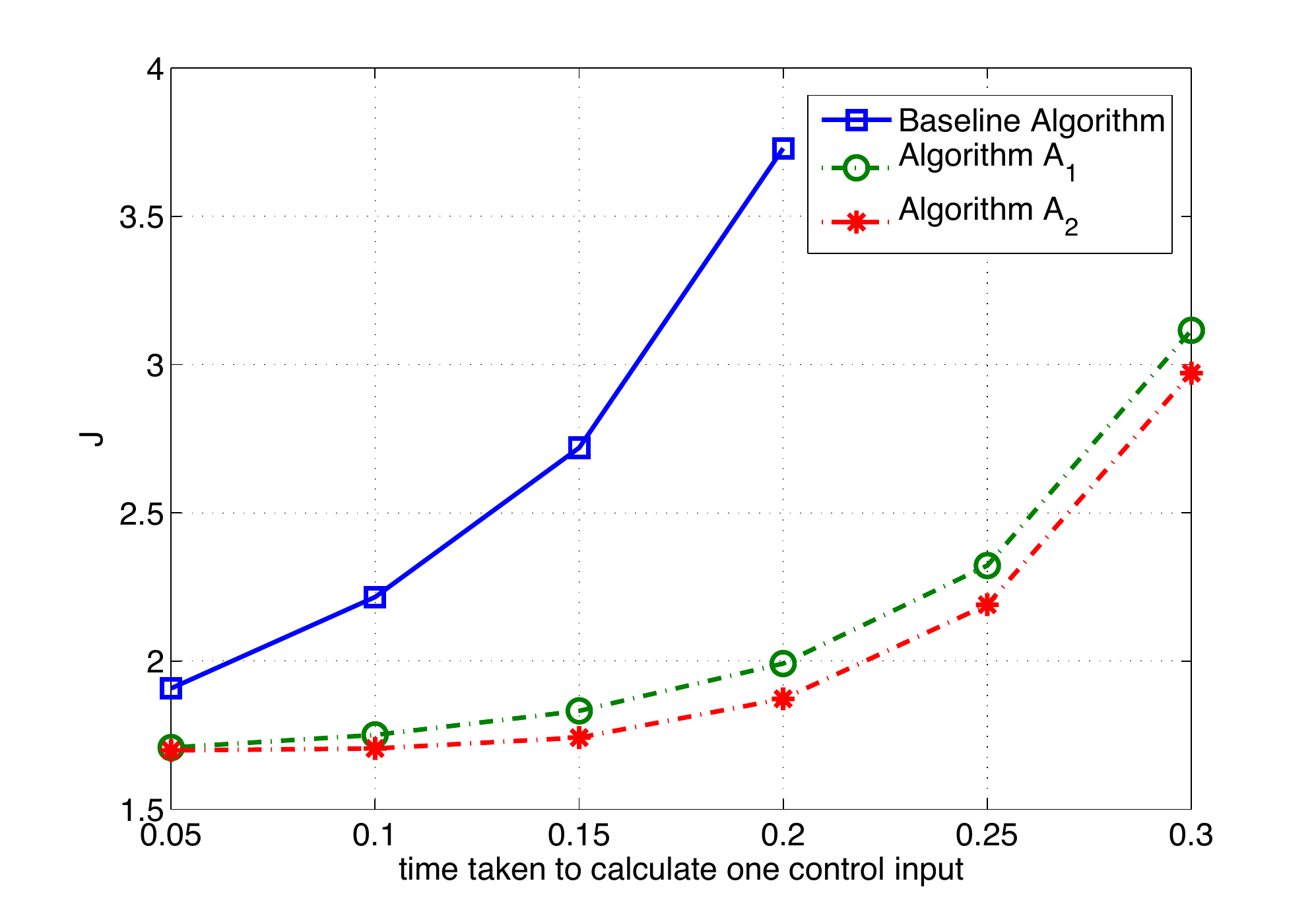}
  \caption{Empirical cost achieved when controlling the nonlinear plant model~(\ref{eq:2}) with
    the proposed anytime algorithms and 
     the baseline algorithm~\eqref{eq:4}, as a function of $\tau$, the execution
    time required 
    to calculate 
    one control input, see~\eqref{eq:78}.}
  \label{fig:1}
\end{figure}



\par As the plant model becomes more open-loop unstable, the proposed algorithms can
be expected to give higher performance gains. Figure~\ref{fig:2} illustrates this intuitive effect for the linear
model
\begin{equation}
\label{eq:39}
  x(k+1)=a x(k)+u(k)+w(k),
\end{equation}
with system parameter $a\in [0.5,1.5]$, and where  $w(k)$ is i.i.d,
Gaussian with zero mean and variance 
0.1. The  policy $\kappa$ is taken
as the associated LQR control law; $\{N\}_{\N_0}$ is distributed as in~\eqref{eq:78} with $\tau=0.3$.  The percentage improvement is plotted for
algorithms $\textsc{A}_{1}$ 
and $\textsc{A}_{2}$,
as compared to the baseline algorithm~\eqref{eq:4}. 
\begin{figure}[t]
  \centering
  \includegraphics[scale=0.42]{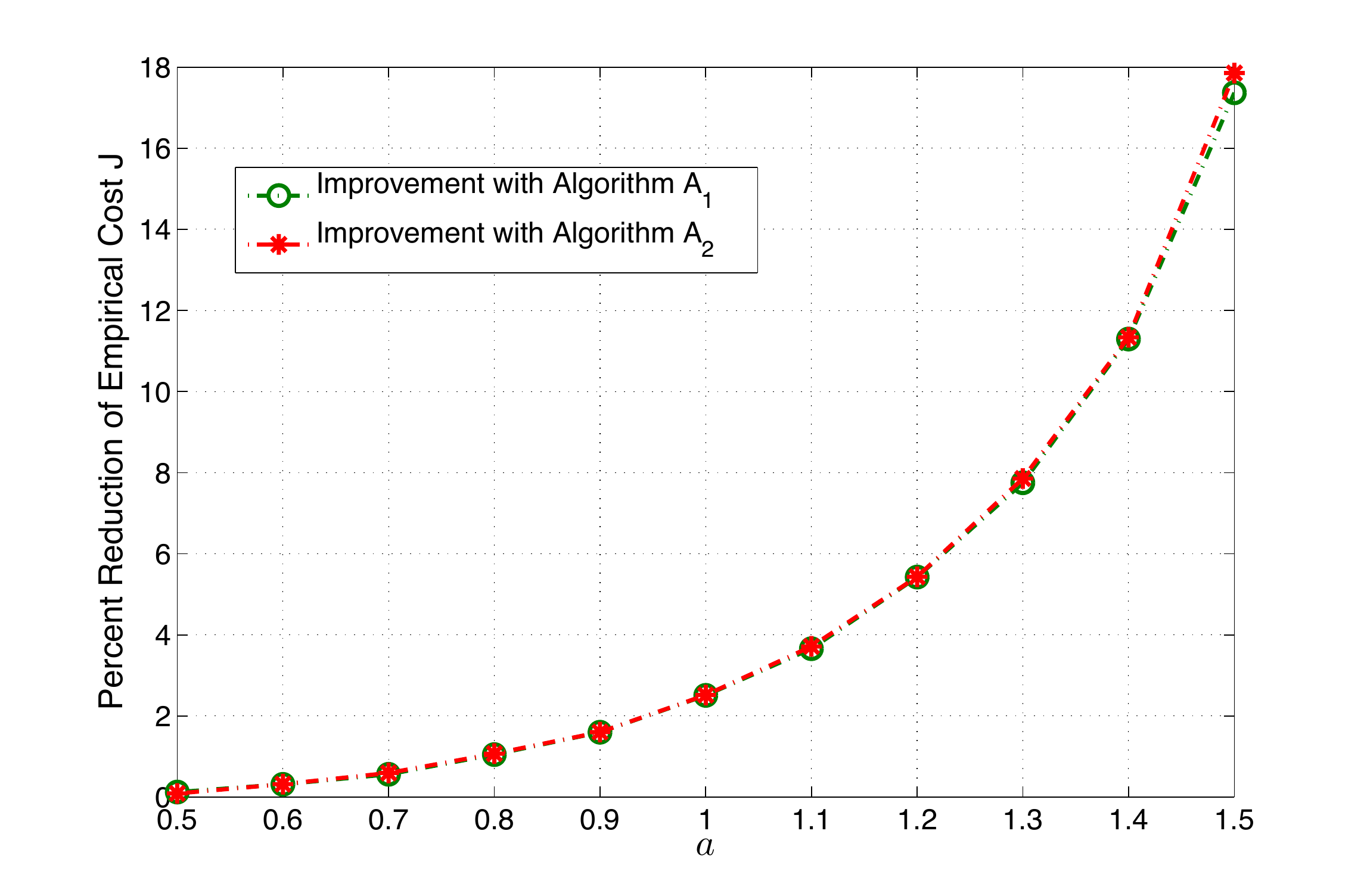}
  \caption{Performance improvement using algorithms A$_1$ and A$_2$ when compared to the
    baseline algorithm~\eqref{eq:4}, for the  model~\eqref{eq:39}. }
  \label{fig:2}
\end{figure}

\par We finally examine the effect of artificially limiting the
maximum buffer size. In particular, if the buffer size is taken as $1$, then one
recovers the baseline algorithm~\eqref{eq:4}; as noted in Section~\ref{sec:algorithm-a_2},
with size 2, Algorithms A$_1$ and A$_2$ are equivalent. Fig.~\ref{fig:3}
illustrates empirical results for a linear plant~\eqref{eq:39} with $a=1.7$. The
processor availability is as per~\eqref{eq:78} with $\tau = 0.23$, thus,
$p_0=p_1=p_2 = p_3=0.23$ and $p_4=0.08$.  Allowing the buffer
size to be of size 4 gives the best results, although a buffer of size 3
gives almost optimal performance.

\begin{figure}[t]
  \centering
  \includegraphics[scale=0.48]{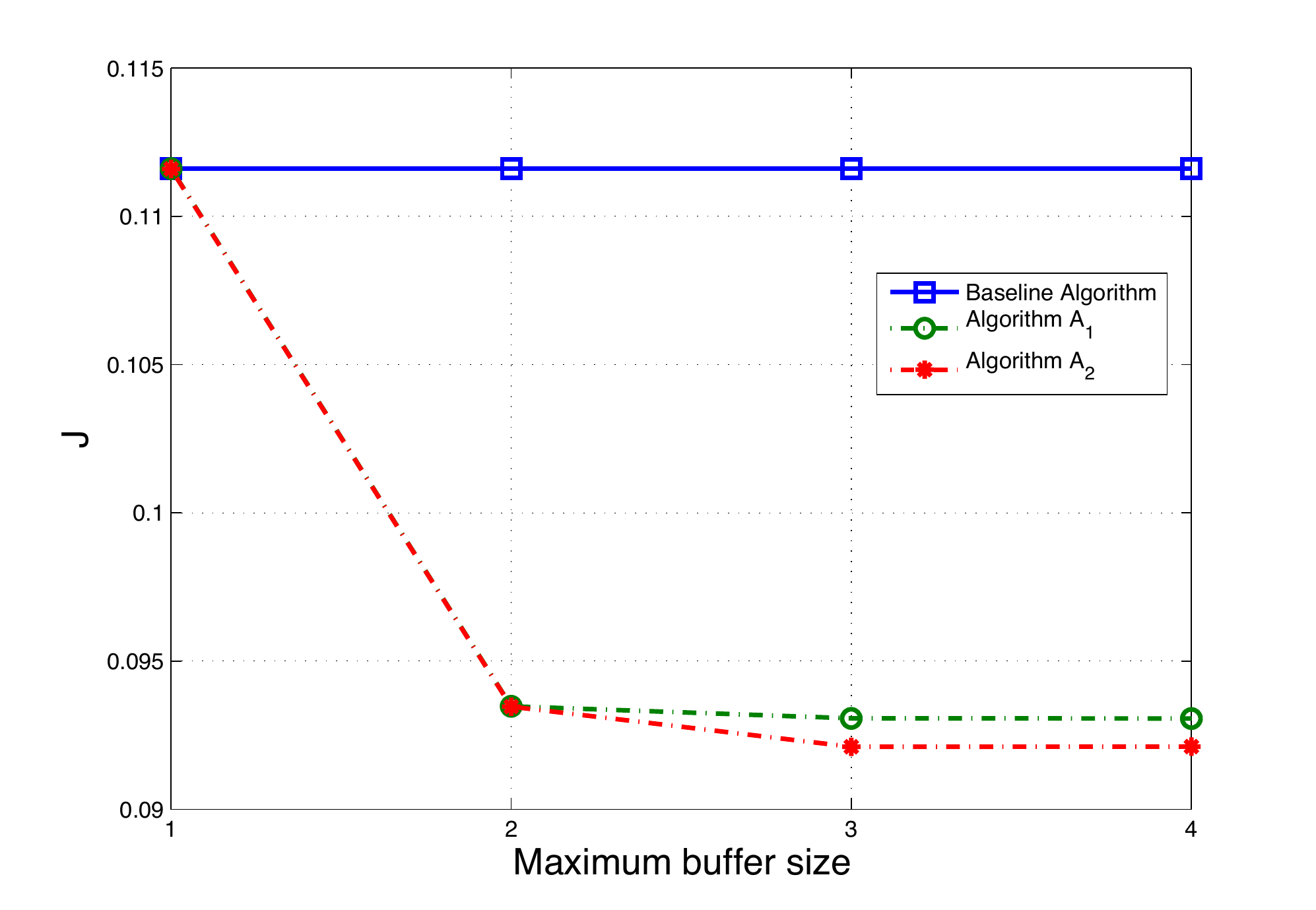}
  \caption{Effect of limiting the maximum buffer size  for the
    model~\eqref{eq:39} with $a=1.7$.}
  \label{fig:3}
\end{figure}

\section{Conclusions}
\label{sec:conclusions}
We proposed two related anytime control algorithms for general nonlinear
processes. The algorithms use available processing resources to compute
sequences of tentative control inputs. 
Thus, even if the processor does not provide sufficient resources at some time steps, the
effect can be partially compensated for. For general non-linear systems, we established
sufficient conditions for stochastic stability. Simple numerical examples indicate that the performance
gains with the 
proposed algorithms can be significant, when compared to a simple baseline algorithm.
 Future work could include examining situations where system assumptions hold only
locally, using the stability and performance
characterizations obtained for processor scheduling, and the development of
anytime algorithms for distributed systems. 



\paragraph*{Acknowledgements}
The authors would like to thank the anonymous reviewers for their valuable comments and suggestions to improve the paper. Research
    supported for the first author under Australian Research Council's Discovery Projects funding  scheme (project number DP0988601) and in part for the second author  by NSF awards 0846631 and
    0834771.

\bibliographystyle{IEEEtran}


%

\begin{biography}[{\includegraphics[width=1in,height=1.25in,clip,keepaspectratio]{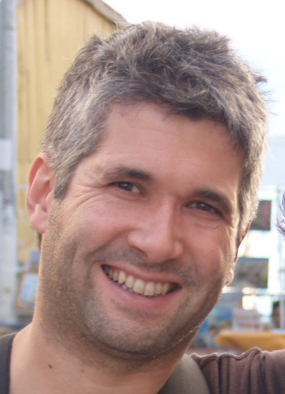}}]{\bf Daniel Quevedo} (S'97--M'05) received Ingeniero Civil Electr\'onico
and Magister en Ingenier\'{\i}a Electr\'onica degrees from the Universidad
T\'ecnica Federico Santa Mar\'{\i}a, Valpara\'{\i}so, Chile in 2000.  In
2005, he received the Ph.D.~degree from The University of Newcastle,
Australia, where he is currently an Associate Professor.  He has been a
visiting researcher at various institutions, including Uppsala University,
Sweden, KTH Stockholm, Sweden, Aalborg University, Denmark, Kyoto University,
Japan, and INRIA Grenoble, France.

\par Dr.\ Quevedo was supported by a full scholarship from the alumni
association during his time at the Universidad 
T\'ecnica Federico Santa Mar\'{\i}a and received several university-wide
prizes upon graduating. He received the IEEE Conference on Decision and
Control Best Student Paper Award in 2003 and was also a finalist  in 
2002.  In
2009, he was awarded an Australian Research Fellowship.   His  research
interests include several areas of automatic control, signal processing, and power
electronics. 
\end{biography}

\begin{biography}[{\includegraphics[width=1in,height=1.25in,clip,keepaspectratio]{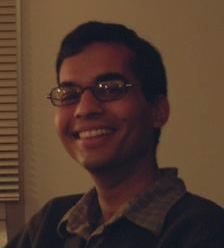}}]{Vijay
    Gupta} 
Vijay Gupta is an Assistant Professor in the Department of Electrical Engineering at the University of Notre Dame. He received his B. Tech degree from the Indian Institute of Technology, Delhi and the M.S. and Ph.D. degrees from the California Institute of Technology, all in Electrical Engineering. He has served as a research associate in the Institute for Systems Research at the University of Maryland, College Park, and as a consultant to the Systems Group at the United Technology Research Center, Hartford, CT.  His research interests include various topics at the interaction of communication, computation and control. He received the NSF Career award in 2009.
\end{biography}

\end{document}